\documentclass{article}

    \usepackage[final]{neurips_2024}

\usepackage[utf8]{inputenc} 
\usepackage[T1]{fontenc}    
\usepackage{hyperref}       
\usepackage{url}            
\usepackage{booktabs}       
\usepackage{amsfonts}       
\usepackage{nicefrac}       
\usepackage{microtype}      
\usepackage{xcolor}         
\usepackage{graphicx,psfrag,amsmath,amsfonts,verbatim}
\usepackage[small,bf]{caption}
\usepackage{hyperref}       
\usepackage{url}            
\usepackage{booktabs}       
\usepackage{amsfonts}       
\usepackage{nicefrac}       
\usepackage{microtype}      
\usepackage{xcolor}         

\usepackage{amsmath}
\usepackage{amsthm}
\usepackage{thmtools}
\usepackage{thm-restate}
\newtheorem{theorem}{Theorem}
\newtheorem{lemma}[theorem]{Lemma}

\theoremstyle{definition}
\newtheorem{definition}{Definition}

\newtheorem{assumption}{Assumption}

\usepackage{subfigure}
\usepackage{caption}
\usepackage{algorithm}
\usepackage{algorithmic}
\usepackage{graphicx}
\usepackage{thm-restate}
\usepackage{caption}
\usepackage{booktabs}
\usepackage{wrapfig}
\usepackage{multirow}

\title{Online Control with Adversarial Disturbance for\\ Continuous-time Linear Systems}

\author{%
  Jingwei Li \\
  IIIS, Tsinghua University\\
  Shanghai Qizhi Institute \\
\texttt{ljw22@mails.tsinghua.edu.cn} \\
  \And
  Jing Dong \\
  The Chinese University of Hong Kong, Shenzhen \\
\texttt{jingdong@link.cuhk.edu.cn} \\
  \AND
  Can Chang \\
  IIIS, Tsinghua University \\
\texttt{cc22@mails.tsinghua.edu.cn} \\
  \And
  Baoxiang Wang \\
  The Chinese University of Hong Kong, Shenzhen \\
\texttt{bxiangwang@cuhk.edu.cn} \\
  \And
  Jingzhao Zhang \\
  IIIS, Tsinghua University \\
  Shanghai Qi zhi Institute \\
  \texttt{jingzhaoz@mail.tsinghua.edu.cn} \\
}

\begin{document}

\maketitle

\begin{abstract}
We study online control for continuous-time linear systems with finite sampling rates, where the objective is to design an online procedure that learns under non-stochastic noise and performs comparably to a fixed optimal linear controller. 
We present a novel two-level online algorithm, by integrating a higher-level learning strategy and a lower-level feedback control strategy. This method offers a practical and robust solution for online control, which achieves sublinear regret. Our work provides the first nonasymptotic results for controlling continuous-time linear systems with finite number of interactions with the system. Moreover, we examine how to train an agent in domain randomization environments from a non-stochastic control perspective. By applying our method to the SAC (Soft Actor-Critic) algorithm, we achieved improved results in multiple reinforcement learning tasks within domain randomization environments. Our work provides new insights into non-asymptotic analyses of controlling continuous-time systems. Furthermore, our work brings practical intuition into controller learning under non-stochastic environments.
\end{abstract}

\section{Introduction}

A major challenge in robotics is to deploy simulated controllers into real-world. This process, known as sim-to-real transfer, can be difficult due to misspecified dynamics, unanticipated real-world perturbations, and non-stationary environments. Various strategies have been proposed to address these issues, including domain randomization, meta-learning, and domain adaptation~\cite{hofer2021sim2real,chenunderstanding, hu2022provable}.
Although they have shown great effectiveness in experimental results, training agents within these  setups poses a significant challenge. To accommodate different environments, the strategies developed by agents tend to be overly conservative~\cite{mehta2020active, amiranashvili2021pre} or lead to suboptimal outcomes~\cite{vuong2019pick,mozian2020learning}.

In this work, we provide an analysis of the sim-to-real transfer problem from an online control perspective. Online control focuses on iteratively updating the controller after deployment (i.e., online) based on collected trajectories. Significant progress has been made in this field by applying insights from online learning to linear control problems~\cite{abbasi2011regret, abbasi2014tracking, cohen2018online, hazan2020nonstochastic, chen2021black, basei2022logarithmic, andrew2013tale, goel2019online}.

Following this line of work, we approach the sim-to-real transfer issue for continuous-time linear systems as a non-stochastic control problem, as explored in previous works~\cite{hazan2020nonstochastic, chen2021black, basei2022logarithmic}. These studies provide regret bounds for an online controller that lacks prior knowledge of system perturbations. However, a gap remains as no previous analysis has specifically investigated continuous-time systems, but real world systems often evolve continuously in time.

Existing literature on online continuous control is limited~\cite{vrabie2009adaptive, jiang2012computational, duncan1992least,
rizvi2018output}. Most continuous control research emphasizes the development of model-free algorithms, such as policy iteration, under the assumption of noise absence. Recently, \cite{basei2022logarithmic} examined online continuous-time linear quadratic control problem and achieves sublinear regret. However, it relies on the assumption of standard Brownian noise instead of non-stochastic noise that may not always hold true in real-world applications. This leads us to the crucial question:

\begin{center}
    \textit{Is it possible to design an online non-stochastic control algorithm \\in a continuous-time setting that achieves sublinear regret?}
\end{center}

Our work addresses this question  by proposing a two-level online controller. The higher-level controller symbolizes the policy learning process and updates the policy at a low frequency to minimize regret. Conversely, the lower-level controller delivers high-frequency feedback control input to reduce discretization error. Our proposed algorithm results in regret bounds for continuous-time linear control in the face of non-stochastic disturbances.

Furthermore, we implement the ideas from our theoretical analysis and test them in several experiments. Note that the key difference between our algorithm and traditional online policy optimization is that we utilize information from past states with some skips to enable faster adaptation to environmental changes. Although the aforementioned concepts are often adopted experimentally as frame stacking and frame skipping, there is relatively little known about the appropriate scenarios for applying these techniques. Our analysis and experiments demonstrate that these techniques are particularly effective in developing adaptive policies for uncertain environments. We choose the task of training agents in a domain randomization environment to evaluate our method, and the results confirm that these techniques substantially improve the agents' performance.

\section{Related Works}
The control theory of linear dynamical systems under disturbances has been thoroughly examined in various contexts, such as linear quadratic stochastic control~\cite{athans1971role}, robust control~\cite{stengel1994optimal,khalil1996robust}, system identification~\cite{goodwin1981discrete, kumar1983optimal, campi1998adaptive, ljung1998system}. However, most of these problems are investigated in non-robust settings, with robust control being the sole exception where adversarial perturbations in the dynamic are permitted. In this scenario, the controller solves for the optimal linear controller in the presence of worst-case noise. Nonetheless, the algorithms designed in this context can be overly conservative as they optimize over the worst-case noise, a scenario that is rare in real-world applications. We will elaborate on the difference between robust control and online non-stochastic control in Section 3.

\paragraph{Online Control}
There has been a recent surge of interest in online control, as demonstrate by studies such as~\cite{abbasi2011regret, abbasi2014tracking, cohen2018online}. In online control, the player interacts with the environment and updates the policy in each round aiming to achieve sublinear regret. In scenarios with stochastic Gaussian noise, \cite{cohen2018online} provides the first efficient algorithm with an $O(\sqrt{T})$ regret bound. However, in real-world applications, the assumption of Gaussian distribution is often unfulfilled.

\cite{agarwal2019online} pioneers research on non-stochastic online control, where the noises can be adversarial. Under general convex costs, they introduce the Disturbance-Action Policy Class. Using an online convex optimization (OCO) algorithm with memory, they achieve an $O(\sqrt{T})$ regret bound. Subsequent studies extend this approach to other scenarios, such as quadratic costs~\cite{basei2022logarithmic}, partial observations~\cite{simchowitz2020improper, simchowitz2020making} or unknown dynamical systems~\cite{hazan2020nonstochastic,chen2021black}. Other works yield varying theoretical guarantees like online competitive ratio~\cite{goel2022best, shi2020online}.

\paragraph{Online Continuous Control}
Compared to online control, there has been relatively little research on model-based continuous-time control. Most continuous control works focus on developing model-free algorithms such as policy iteration (e.g.~\cite{vrabie2009adaptive, jiang2012computational, rizvi2018output}), typically assuming zero-noise. This is because analyzing the system when transition dynamics are represented by differential equations, rather than recurrence formulas, poses a significant challenge.

Recently, \cite{basei2022logarithmic} studies online continuous-time linear quadratic control with standard Brownian noise and unknown system dynamics. They propose an algorithm based on the least-square method, which estimates the system's coefficients and solves the corresponding Riccati equation. The papers~\cite{shirani2022thompson, faradonbeh2023online} also focus on online control setups with continuous-time stochastic linear systems and unknown dynamics. They achieve $O(\sqrt{T}\log T)$ regret by different approaches. \cite{shirani2022thompson} uses the Thompson sampling algorithm to learn optimal actions. \cite{faradonbeh2023online} takes a randomized-estimates policy to balance exploration and exploitation.
The main difference between \cite{basei2022logarithmic, shirani2022thompson, faradonbeh2023online} and our paper is that they consider stochastic noise of Brownian motion which can be quite stringent and may fail in real-world applications, while the noise in our setup is non-stochastic. This makes our analysis completely different from theirs. 

\paragraph{Domain Randomization}
Domain randomization, which is proposed by~\cite{tobin2017domain}, is a commonly used technique for training agents to adapt to different (real) environments by training in randomized simulated environments. From the empirical perspective, many previous works focus on designing efficient algorithms for learning in a randomized simulated environment (by randomizing environmental settings, such as friction coefficient) such that the algorithm can adapt well in a new environment, \cite{muratore2019assessing, zakharov2019deceptionnet,mehta2020active, muratore2021data,muratore2022neural}. Other works study how to effectively randomize the simulated environment so that the trained algorithm would generalize well in other environments~\cite{vuong2019pick,mozian2020learning,tiboni2023dropo}. However, prior research has not explored how to apply certain theoretical analysis ideas to train agents in domain-randomized environments. Limited previous works, such as~\cite{chenunderstanding} and~\cite{hu2022provable}, concentrate on theoretically analyzing the sim-to-real gap within specific domain randomization models but they do not test their algorithms in real domain randomization environments.

\section{Problem Setting}

In this paper, we consider the online non-stochastic control for continuous-time linear systems. Therefore, we provide a brief overview below and define our notations.

\subsection{Continuous-time Linear Systems}
The Linear Dynamical System can be considered a specific case of a continuous Markov decision process with linear transition dynamics. The state transitions are governed by the following equation:
$$
\dot{x}_{t}=A x_t+B u_t+w_t \,,
$$
where $x_t$ is the state at time $t$, $u_t$ is the action taken by the controller at time $t$, and $w_t$ represents the disturbance at time $t$. Follow the setup of~\cite{agarwal2019online}, we assume $x_0=0$. We do not make any strong 
 assumptions about the distribution of $w_t$, and we also assume that the distribution of $w_t$ is unknown to the learner beforehand. This implies that the disturbance sequence $w_t$ can be selected adversarially.

When the action $u_t$ is applied to the state $x_t$, a cost $c_t(x_t, u_t)$ is incurred. Here, we assume that the cost function $c_t$ is convex. However, this cost is not known in advance and is only revealed after the action $u_t$ is implemented at time $t$. In the system described above, an online policy $\pi$ is defined as a function that maps known states to actions, i.e., $u_t = \pi(\{x_\xi | \xi \in [0,t]\})$. Our goal, then, is to design an algorithm that determines such an online policy to minimize the cumulative cost incurred. Specifically, for any algorithm $\mathcal{A}$, the cost incurred over a time horizon $T$ is:
$$
J_T(\mathcal{A})=\int_0^T c_t(x_t, u_t) dt \,.
$$
In scenarios where the policy is linear (i.e., a linear controller), such that $u_t=-K x_t$, we use $J(K)$ to denote the cost of a policy $K \in \mathcal{K}$ from a certain class $\mathcal{K}$.

\subsection{Difference between Robust and Online Non-stochastic Control}
While both robust and online non-stochastic control models incorporate adversarial noise, it's crucial to understand that their objectives differ significantly.

The objective function for robust control, as seen in~\cite{stengel1994optimal,khalil1996robust}, is defined as:
$$
\min_{u_1} \max_{w_{1: T}} \min_{u_2} \ldots \min_{u_t} \max_{w_T} J_T(\mathcal{A}) \,,
$$
Meanwhile, the objective function for online non-stochastic control, as discussed in~\cite{agarwal2019online}, is: 
$$
\min_{\mathcal{A}} \max_{w_{1: T}} (J_T(\mathcal{A})-\min_{K \in \mathcal{K}} J_T(K))  \,.
$$

Note that the robust control approach seeks to directly minimize the cost function, while online non-stochastic control targets the minimization of regret, which is the discrepancy between the actual cost and the cost associated with a baseline policy. Additionally, in robust control, the noise at each step can depend on the preceding policy, whereas in online non-stochastic control, all the noise is predetermined (though unknown to the player).
\subsection{Assumptions}
We operate under the following assumptions throughout this paper. To be concise, we denote $\|\cdot\|$ as the $L_2$ operator norm of the vector and matrix. Firstly, we make assumptions concerning the system dynamics and noise: 

\begin{assumption}\label{asmp:bounded_dynamic_matrix}
The matrices that govern the dynamics are bounded, meaning $\|A\| \leq \kappa_A$ and $\|B\| \leq \kappa_B$, where $\kappa_A$ and $\kappa_B$ are constants. Moreover, the perturbation and its derivative are both continuous and bounded: $\left\|w_t\right\| , \left\|\dot{w}_t\right\| \leq W$, with $W$ being a constant.
\end{assumption}

These assumptions ensure that we can bound the states and actions, as well as their first and second-order derivatives.
Next, we make assumptions regarding the cost function:

\begin{assumption} \label{asmp:bounded_cost}
The costs $c_t(x, u)$ are convex in $x$ and $u$. Additionally, if there exists a constant $D$ such that $\|x\|,\|u\| \leq D$, then we have the following inequalities of the costs:
$
\left|c_t(x, u)\right| \leq \beta D^2,\left\|\nabla_x c_t(x, u)\right\|,\left\|\nabla_u c_t(x, u)\right\| \leq G D
$, $|c_{t_1}(x,u)-c_{t_2}(x,u)| \le L|t_1-t_2|D^2$,
\end{assumption}

where $\beta$,$G$ and $L$ are constants corresponding to the cost function. This assumption implies that if the differences between states and actions are small, then the error in their cost will also be relatively small.

\subsection{Strongly Stable Policy}
We next describe our baseline policy class introduced in~\cite{cohen2018online}. Note that the continuous system and the discrete system are different. If we consider the approximation over a relatively small interval $h$, we get 
\begin{align*}
    x_{t+h} = x_t + \int^{t+h}_{s=t}\dot{x}_sds = &x_t + \int^{t+h}_{s=t}Ax_s + Bu_s + w_sds \\
    \approx& x_t + h(Ax_t+Bu_t + w_t) = (I+hA)x_t+hBu_t+hw_t \,.
\end{align*}
Therefore, if we consider the transition of a discrete system $x_{i+1} = \Tilde{A}x_i + \Tilde{B}u_i + \Tilde{w}_i$, we get the approximation $\Tilde{A} \approx I+hA$, $\Tilde{B} \approx hB$. Hence, we extend the definition of a strongly stable policy~\cite{cohen2018online, agarwal2019online} in the discrete system to the continuous system as follows:
\begin{definition}\label{def:linear_policy}
A linear policy $K$ is $(\kappa, \gamma)$-strongly stable if, for any $h > 0$ that is sufficiently small, there exist matrices $L_h, P$ such that $I + h(A-B K)= P L_h P^{-1} $, with the following two conditions:
\begin{enumerate}
\item The norm of $L_h$ is strictly smaller than unity and dependent on $h$, i.e., $\|L_h\| \leq 1-h\gamma$.
\item The controller and transforming matrices are bounded, i.e., $\|K\| \leq \kappa$ and $\|P\|,\left\|P^{-1}\right\| \leq \kappa$.
\end{enumerate}
\end{definition}

The above definition ensures the system can be stabilized by a linear controller $K$. 

\subsection{Regret Formulation}
To evaluate the designed algorithm, we follow the setup in~\cite{cohen2018online, agarwal2019online} and use regret, which
is defined as the cumulative difference between the cost
incurred by the policy of our algorithm and the cost incurred
by the best policy in hindsight. Let $\mathcal{K}$ denotes the class of strongly stable linear policies, i.e. $\mathcal{K}=\{K: K$ is $(\kappa, \gamma)$-strongly stable$\}$. Then we try to minimize the regret of algorithm:
$$
\min_{\mathcal{A}} \max_{w_{1: T}} \mathrm{Regret}(\mathcal{A}) = \min_{\mathcal{A}} \max_{w_{1: T}}(J_T(\mathcal{A})-\min _{K \in \mathcal{K}} J_T(K)) \,.
$$

\section{Algorithm Design}
In this section, we outline the design of our algorithm and formally define the concepts involved in deriving our main theorem. We summarize our algorithm design as follows:

First, we discretize the total time period $T$ into smaller intervals of length $h$. We use the information at each point $x_h, x_{2h}, \ldots$ and $u_h, u_{2h}, \ldots$ to approximate the actual cost of each time interval, leveraging the continuity assumption. This process does introduce some discretization errors.

Next, we employ the Disturbance-Action policy (DAC)~\cite{agarwal2019online}. This policy selects the action based on the current time step and the estimations of disturbances from several past steps. This policy can approximate the optimal linear policy in hindsight when we choose suitable parameters. However, the optimal policy $K^\ast$ is unknown, so we cannot directly acquire the optimal choice. To overcome this, we employ the OCO with memory framework~\cite{anava2015online} to iteratively adjust the DAC policy parameter $M_t$ to approximate the optimal solution $M^*$.

After that, we introduce the concept of the ideal state $y_t$ and ideal action $v_t$ that approximate the actual state $x_t$ and action $u_t$. Note that both the state and policy depend on all DAC policy parameters $M_1, M_2, \ldots, M_t$. Yet, the OCO with memory framework only considers the previous $H$ steps. Therefore, we need to consider ideal state and action. $y_t$ and $v_t$ represent the state the system would reached if it had followed the DAC policy $\left\{M_{t-H}, \ldots, M_t\right\}$ at all time steps from $t-H$ to $t$, under the assumption that the state $x_{t-H}$ was $0$.

From all the analysis above, we can decompose the regret as three parts: the discretization error $R_1$, the regret of the OCO with memory $R_2$, and the approximation error between the ideal cost and the actual cost $R_3$.

Then we will formally introduce out method and define all the concepts. In the subsequent discussion, we use shorthand notation to denote the cost, state, control, and disturbance variables $c_{ih}$, $x_{ih}$, $u_{ih}$, and $w_{ih}$ as $c_i$, $x_i$, $u_i$, and $w_i$, respectively.

First, we need to define the Disturbance-Action Policy Class(DAC) for continuous systems:
\begin{definition}\label{def:dac}
The Disturbance-Action Policy Class(DAC) is defined as:
    \begin{align*}
        u_t=-K x_t+\sum_{i=1}^{l} M^{i}_t \hat{w}_{t-i} \,,
    \end{align*}
where $K$ is a fixed strongly stable policy, $l$ is a parameter that signifies the dimension of the policy class, $M_t = \{M^1_t, \ldots , M^l_t\}$ is the weighting parameter of the disturbance at step $t$, and $\hat{w}_{t}$ is the estimated disturbance:
\begin{align}\label{equ:noise}
    \hat{w}_t = \frac{x_{t+1} - x_t - h(Ax_t + Bu_t)}{h} \,.
\end{align}
\end{definition}
We note that this definition differs from the DAC policy in discrete systems~\cite{agarwal2019online} as we utilize the estimation of disturbance over an interval $[t, t+h]$ instead of only the noise in time $t$. It counteracts the second-order residue term of the Taylor expansion of $x_t$ and is also an online policy as it only requires information from the previous state.

Our higher-level controller adopts the OCO with memory framework. A technical challenge lies in balancing the approximation error and OCO regret. To achieve a low approximation error, we desire the policy update interval $H$ to be inversely proportional to the sampling distance $h$. However, this relationship may lead to large OCO regret. To mitigate this issue, we introduce a new parameter $m = \Theta(\frac{1}{h})$, representing the lookahead window. We update the parameter $M_t$ only once every $m$ iterations, further reducing the OCO regret without negatively impacting the approximation error:
\begin{align*}\label{equ:OCO}
    M_{t+1} = \left\{
		\begin{aligned}
		& \Pi_{\mathcal{M}}\left(M_t-\eta \nabla g_t(M)\right) & &{\text{if }t \text{ mod }  m == 0} \,, \\
		&M_t & & {\text{otherwise}} \,.
		\end{aligned}
		\right.
\end{align*}
Where $g_t$ is a function corresponding to the loss function $c_t$ and we will introduce later in Algorithm~\ref{alg}.
For notational convenience and to avoid redundancy, we denote $\Tilde{M}_{[t/m]} = M_{t}$. We can then define the ideal state and action. Due to the properties of the OCO with memory structure, we need to consider only the previous $Hm$ states and actions, rather than all states. As a result, we introduce the definition of the ideal state and action.
During the interval $t \in [im, (i+1)m - 1]$, the learning policy remains unchanged, so we could define the ideal state and action follow the definition in~\cite{agarwal2019online}:

\begin{definition}\label{def:ide-state}
The ideal state $y_t$ and action $v_t$ at time $t\in [im, (i+1)m - 1]$ are defined as 
    \begin{align*}
    y_t = x_t(\Tilde{M}_{i-H}, ..., \Tilde{M}_i), v_t=-K y_t+\sum_{j=1}^{l} M_{i}^j w_{t-i} \,.
\end{align*}
where the notation indicates that we assume the state $x_{t-H}$ is $0$ and that we apply the DAC policy $\left(\Tilde{M}_{i-H}, \ldots, \Tilde{M}_i\right)$ at all time steps from $t-Hm$ to $t$.
\end{definition}
We can also define the ideal cost in this interval follow the definition in~\cite{agarwal2019online}:
\begin{definition}\label{def:ideal}
The ideal cost function during the interval $t \in [im, (i+1)m - 1]$ is defined as follows:
\begin{align*}
    f_i\left(\Tilde{M}_{i-H}, \ldots, \Tilde{M}_i\right)=\sum_{t = im}^{(i+1)m - 1}c_t\left(y_t\left(\Tilde{M}_{i-H}, \ldots, \Tilde{M}_{i}\right), v_t\left(\Tilde{M}_{i-H}, \ldots, \Tilde{M}_i\right)\right) \,.
\end{align*}
\end{definition}

With all the concepts presented above, we are now prepared to introduce our algorithm:
\begin{algorithm}[H]
    \caption{Continuous two-level online control algorithm} \label{alg}
    \begin{algorithmic}
    \STATE \textbf{Input:} step size $\eta$, sample distance $h$, policy update parameters $H, m$, parameters $\kappa, \gamma, T$.
    \STATE Define sample numbers $n = \lceil T/h\rceil$, OCO policy update times $p = \lceil n/m \rceil$.
    \STATE Define DAC policy update class 
    $\mathcal{M} = \left\{\Tilde{M}=\left\{\Tilde{M}^{1} \ldots \Tilde{M}^{Hm}\right\}:\left\|\Tilde{M}^{i}\right\| \leq 2h\kappa^3(1-\gamma)^{i-1}\right\}$.
    \STATE Initialize $M_0 \in \mathcal{M}$ arbitrarily.
    \FOR{$k = 0, \ldots, p-1$}
        \FOR{$s = 0, \ldots, m-1$}
            \STATE Denote the discretization time $r = km + s$. 
            \STATE Use the action $u_t=-K x_r+h\sum_{i=1}^{Hm} \Tilde{M}_{k}^{i} \hat{w}_{r-i}$ during the time period $t\in [rh, (r+1)h]$.
            \STATE Observe the new state $x_{r+1}$ at time $(r+1)h$ and record $\hat{w}_r$ according to Equation (\ref{equ:noise}).
        \ENDFOR 
        \STATE Define the function $g_k(M) = f_k(M,\ldots,M)$.
        \STATE Update OCO policy $\Tilde{M}_{k+1} = \Pi_{\mathcal{M}}\left(\Tilde{M}_k-\eta \nabla g_k(\Tilde{M}_k)\right)$.
    \ENDFOR
    \end{algorithmic}
\end{algorithm}

\section{Main Result}

In this section, we present the primary theorem of online continuous control regret analysis:

\begin{restatable}{theorem}{main}
\label{thm:main}
Under Assumption \ref{asmp:bounded_dynamic_matrix}, \ref{asmp:bounded_cost}, a step size of $\eta =
\Theta(\sqrt{\frac{m}{Th}})$, and a DAC policy update frequency $m = \Theta(\frac{1}{h})$, Algorithm \ref{alg} attains a regret bound of
\begin{align*}
J_T(\mathcal{A})-\min _{K \in \mathcal{K}} J_T(K) \leq O(nh(1-h\gamma)^{\frac{H}{h}}) + O(\sqrt{nh}) + O(Th) \,.
\end{align*}
With the sampling distance $h = \Theta(\frac{1}{\sqrt{T}})$, and the OCO policy update parameter $H = \Theta(\log(T))$, Algorithm \ref{alg} achieves a regret bound of
\begin{align*}
J_T(\mathcal{A})-\min _{K \in \mathcal{K}} J_T(K) \leq O\left(\sqrt{T} \log \left(T\right)\right) \,.
\end{align*}
\end{restatable}

Theorem \ref{thm:main} demonstrates a regret that matches the regret of a discrete system~\cite{agarwal2019online}. Despite the analysis of a continuous system differing from that of a discrete system, we can balance discretization error, approximation error, and OCO with memory regret by selecting an appropriate update frequency for the policy. Here, $O(\cdot)$ and $\Theta(\cdot)$ are abbreviations for the polynomial factors of universal constants in the assumption.

While we defer the detailed proof to the appendix, we outline the key ideas and highlight them below.

\paragraph{Challenge and Proof Sketch}
\begin{wrapfigure}{r}{0.45\textwidth}
  \centering
  \includegraphics[width=0.4\textwidth]{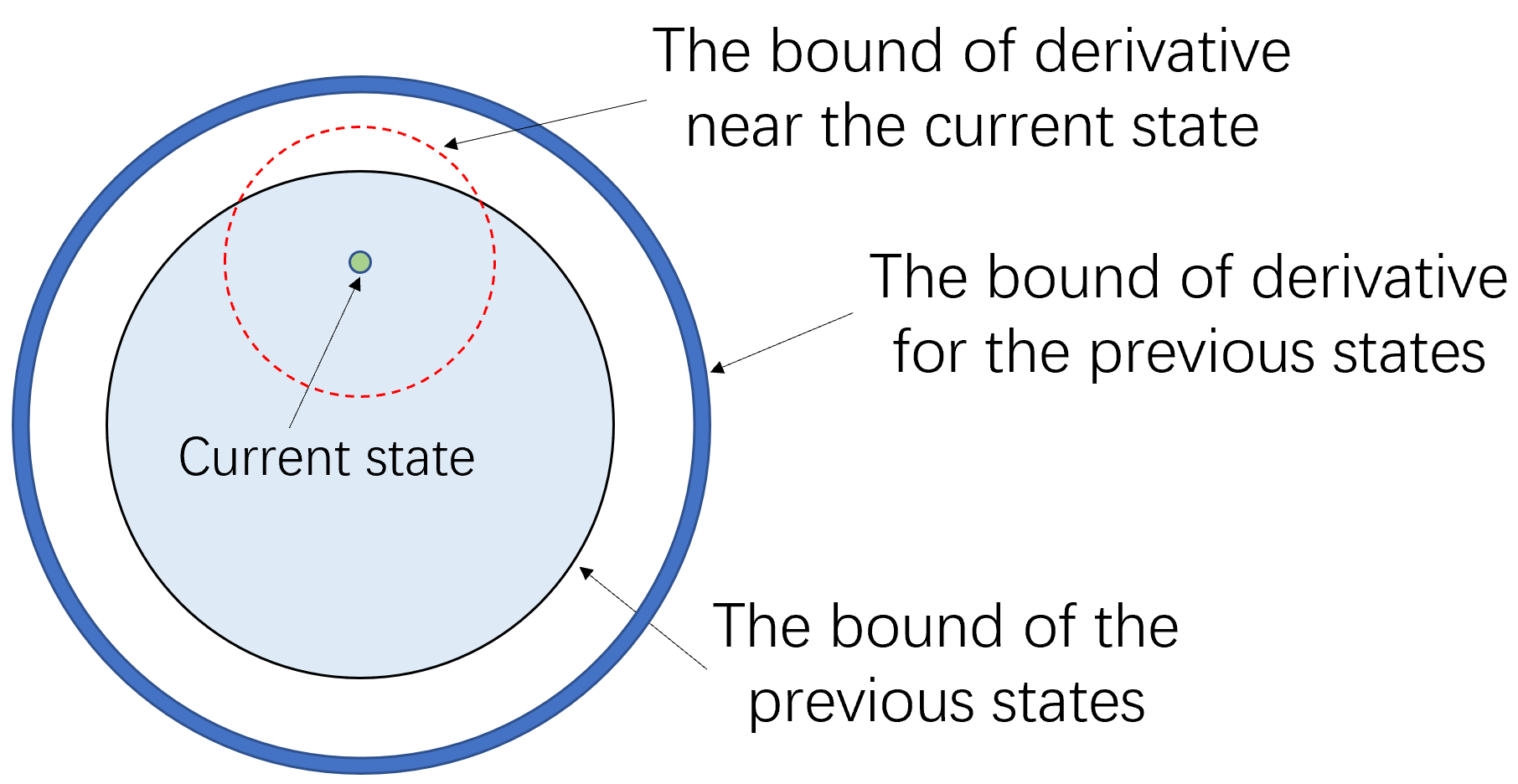}  
  \caption{Bounding the states and their derivatives separately. We employ Gronwall's inequality with the induction method to bound the states.}
    \label{fig:my_label}
\end{wrapfigure}

We first explain why we cannot directly apply the methods for discrete nonstochastic control from~\cite{agarwal2019online} to our work. To utilize Assumption \ref{asmp:bounded_cost}, it is necessary first to establish a union bound over the states.
In a discrete-time system, it can be easily proved by applying the dynamics inequality $\|x_{t+1}\| \le a \|x_t\| + b$ (where $a < 1$) and the induction method presented in~\cite{agarwal2019online}. However, for a continuous-time system, a different approach is necessary because we only have the differential equation instead of the state recurrence formula.

To overcome this challenge, we employ Gronwall's inequality to bound the first and second-order derivatives in the neighborhood of the current state. We then use these bounded properties, in conjunction with an estimation of previous noise, to bound the distance to the next state. Through an iterative application of this method, we can argue that all states and actions are bounded.

Another challenge is that 
we need to discretize the system but we must overcome 
the curse of dimensionality caused by discretization. In continuous-time systems, the number of states is inversely proportional to the discretization parameter $h$, which also determines the size of the OCO memory buffer. Our regret is primarily composed of three components:  the error caused by discretization $R_1$, the regret of OCO with memory $R_2$ and the difference between the actual cost and the approximate cost $R_3$. The discretization error $R_1$ is $O(hT)$, therefore if we achieve $O(\sqrt{T})$ regret, we must choose $h$ no more than $O(\frac{1}{\sqrt{T}})$. 

If we update the OCO with memory parameter at each timestep follow the method in~\cite{agarwal2019online}, we will incur the regret of OCO with memory $R_2 = O(H^{2.5}\sqrt{T})$. The difference between the actual cost and the approximate cost $R_3 = O(T(1-h\gamma)^{H})$. To achieve sublinear regret for the third term, we must choose $H = O(\frac{\log T}{h\gamma})$, but since $h$ is no more than $O(\frac{1}{\sqrt{T}})$, $H$ will be larger than $\Theta(\sqrt{T})$, therefore the second term $R_2$ will definitely exceed $O(\sqrt{T})$.

Therefore, we adjust the frequency of updating the OCO parameters by introducing a new parameter $m$, using a two-level approach and update the OCO parameters once in every $m$ steps. This will incur the third term $R_3 = O(T(1-h\gamma)^{Hm})$ but keep the OCO with memory regret $R_2 = O(H^{2.5}\sqrt{T})$, so we can choose $H = O(\frac{\log T}{\gamma})$ and $m=O(\frac{1}{h})$. Then the term of $R_2$ is $O(\sqrt{T}\log T)$ and we achieve the same regret compare with the discrete system.

\section{Experiments}
\label{sec:experiment}
In this section, we apply our theoretical analysis to the practical training of agents. First we highlight the key difference between our algorithm and traditional online policy optimization. 
\begin{enumerate}
    \item \textit{Stack:} While standard online policy optimization learns the optimal policy from the current state $u_t = \phi(x_t)$, an optimal non-stochastic controller employs the DAC policy as outlined in Definition~\ref{def:dac}. Leveraging information from past states aids the agent in adapting to dynamic environments.
    \item \textit{Skip:} Different from the analysis in~\cite{agarwal2019online}, in a continuous-time system we update the state information every few steps, rather than updating it at every step. This solves the curse of dimensionality caused by discretization in continuous-time system. 
\end{enumerate}
The above inspires us with an intuitive strategy for training agents by stacking past observations with some observations to skip. We denote this as \textit{Stack \& skip} for convenience. \textit{Stack \& skip} is frequently used as a heuristic in reinforcement learning, yet little was known about when and why such a technique could boost agent performance.

How should we evaluate our algorithm in a non-stochastic environment? We opt for learning an optimal policy within a domain randomization environment. In this context, each model's parameters are randomly sampled from a predetermined task distribution. We train policies to optimize performance across various simulated models~\cite{tremblay2018training, muratore2019assessing}.

\begin{wrapfigure}{r}{0.65\textwidth}
  \centering
  \vspace{-0.3cm}
\includegraphics[width=0.6\textwidth]{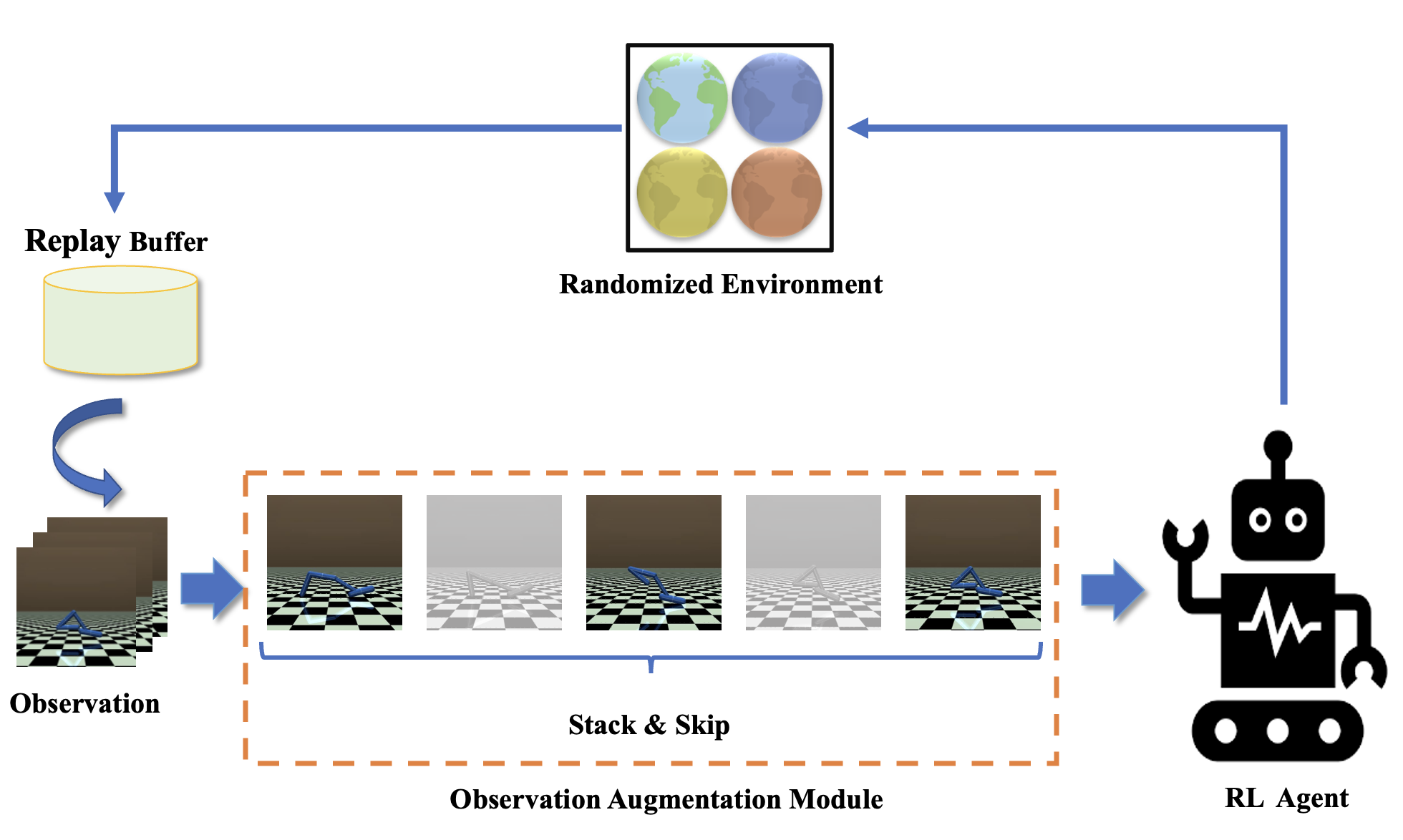}  
  \caption{Leverage past observation of states with some skip.}
    \label{fig:teaser}
\end{wrapfigure}
We observe that learning in Domain Randomization (DR) significantly differs from stochastic or robust learning problems. In DR, sampling from environmental variables occurs at the beginning of each episode, rather than at every step, distinguishing it from stochastic learning where randomness is step-wise independent and identically distributed. This episodic sampling approach allows agents in DR to exploit environmental conditions and adapt to episodic changes within an episode.
On the other hand, robust learning focuses on worst-case scenarios depending on an agent's policy. DR, in contrast, is concerned with the distribution of conditions aimed at broad applicability rather than worst-case perturbations.

In the context of non-stochastic control, the disturbance, while not disclosed to the learner beforehand, remains fixed throughout the episode and does not adaptively respond to the control policy. This setup in non-stochastic control shows a clear parallel to domain randomization: fixed yet unknown disturbances in non-stochastic control mirror the unknown training environments in DR. As the agent continually interacts with these environments, it progressively adapts, mirroring the adaptive process observed in domain randomization.
Therefore, we propose evaluating our algorithm within a domain randomization training task. Subsequently, we introduce the details of our experimental setup:

\begin{wraptable}{r}{0.6\textwidth}
  \centering
  \begin{tabular}{|c|c|c|}
\hline
\textbf{Environment} & \textbf{Parameters}    & \textbf{DR distribution} \\ \hline
\multirow{4}{*}{Hopper} & Joint damping & [0.5, 1.5]                  \\ 
                        & Foot friction       & [1, 3]                   \\
                        & Height of head      & [1.2, 1.7]                 \\
                        & Torso size & [0.025, 0.075]                \\ \hline
\multirow{3}{*}{Half-Cheetah} & Joint damping    & [0.005, 0.015]                  \\
                              & Foot friction       & [3, 7]                   \\
                              & Torso size      & [0.04, 0.06]                 \\ \hline
\multirow{3}{*}{Walker2D} & Joint damping & [0.05, 0.15]                  \\
                        & Density       & [500, 1500]                   \\
                        & Torso size & [0.025, 0.075]                \\ \hline
\end{tabular}
\caption{The DR distributions of environment.}
\label{env_param}
\vspace{-0.1cm}
\end{wraptable}

\paragraph{Environment Setting} 
We conduct experiments on the hopper, half-cheetah, and walker2d benchmarks using the MuJoCo simulator~\cite{Todorov2012MuJoCoAP}. The randomized parameters include environmental physical parameters such as damping and friction, as well as the agent properties such as torso size. We set the range of our domain randomization to follow a distribution with default parameters as the mean value, shown in Table~\ref{env_param}. When training in the domain randomization environment, the parameter is uniformly sampled from this distribution. To analyze the result of generalization, we only change one of the parameters and keep the other parameters as the mean of its distribution in each test environment. We conducted experiments using NVIDIA A40 graphics card.

\paragraph{Algorithm Design and Baseline}
We design a practical meta-algorithm that converts any standard deep RL algorithm into a domain-adaptive algorithm, shown in Figure~\ref{fig:teaser}.
In this algorithm, we augment the original state observation \( o_t^{\mathrm{old}} \) at time \( t \) with past observations, resulting in \( o_t^{\mathrm{new}} = [o_t^{\mathrm{old}}, o_{t-m}^{\mathrm{old}}, \ldots, o_{t-(h-1)m}^{\mathrm{old}}] \). Here \( h \) is the number of past states we leverage and \( m \) is the number of states we skip when we get each of the past states. For clarity in our results, we selected the SAC algorithm for evaluation.
We use a variant of Soft Actor-Critic~(SAC)~\cite{nikishin2022primacy} and leverage past states with some skip as our algorithm. We compare our algorithm with the standard SAC algorithm training on domain randomization environments as our baseline.

\paragraph{Impact of Frame Stack and Frame Skip}
To understand the effects of the frame stack number $h$ and frame skip number $m$, we carried out experiments in the hopper environment with different $h$ and $m$. For each parameter we train with 3 random seeds and take the average. Figure~\ref{fig:frame_stack} shows that the performance increases significantly when the frame stack number is increased from $1$ to $3$, and remains roughly unchanged when the frame stack number continues to climb up. Figure~\ref{fig:frame_skip} shows that the optimal frame skip number is $3$, while both too large or too small frame skip numbers result in sub-optimal results. Therefore, in the following experiments we fix the parameter $h=3$, $m=3$. We train our algorithm with this parameter and standard SAC on hopper and test the performance on more environments. Figure~\ref{tab:hopper} shows that our algorithm outperforms the baseline in all environments.

\begin{figure}[h]
    \centering
    \begin{minipage}[b]{0.49\linewidth} 
        \includegraphics[width=\linewidth]{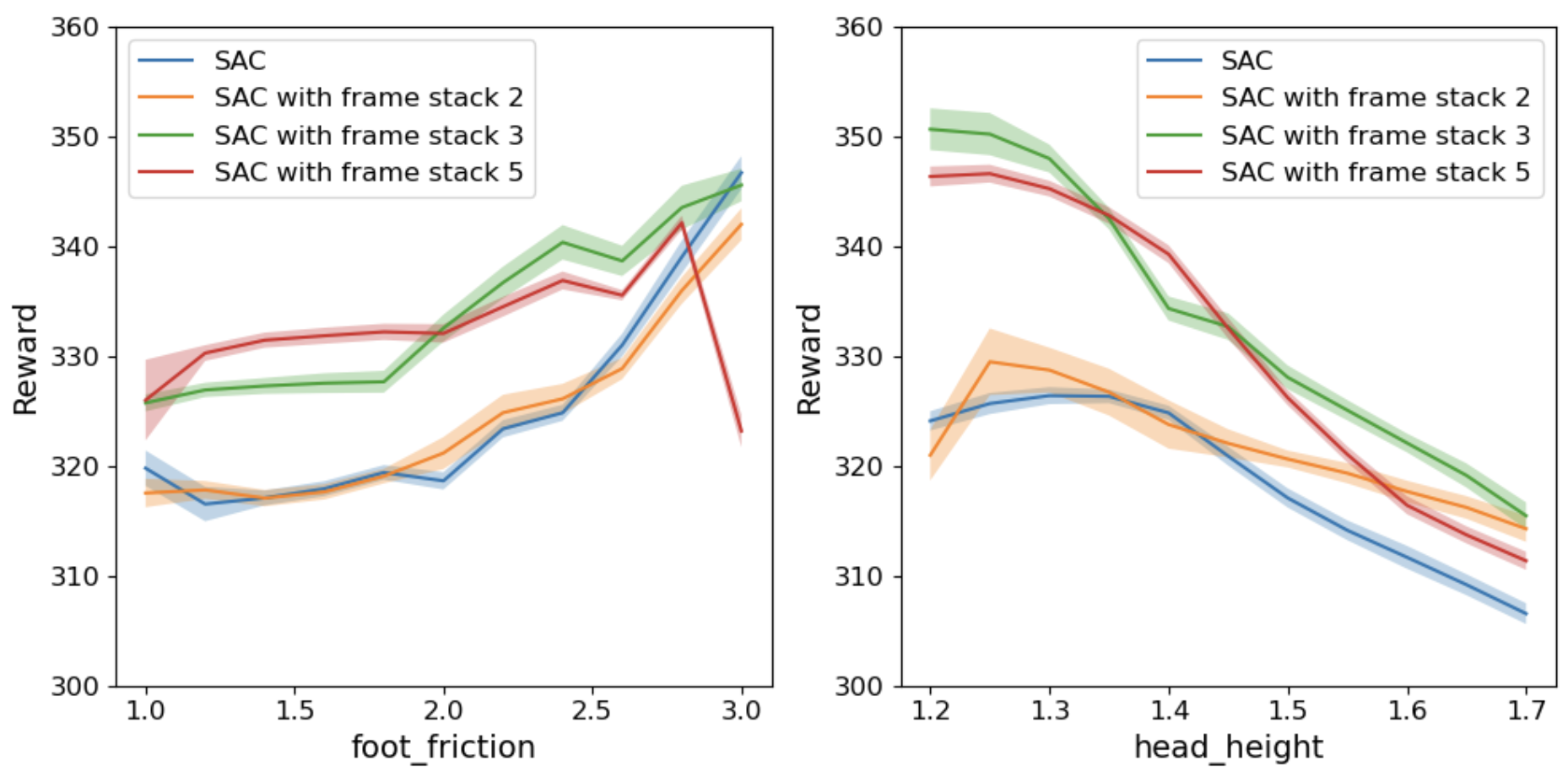}
        \caption{Impact of frame stack number.}
        \label{fig:frame_stack}
    \end{minipage}
    \hfill 
    \begin{minipage}[b]{0.49\linewidth} 
        \includegraphics[width=\linewidth]{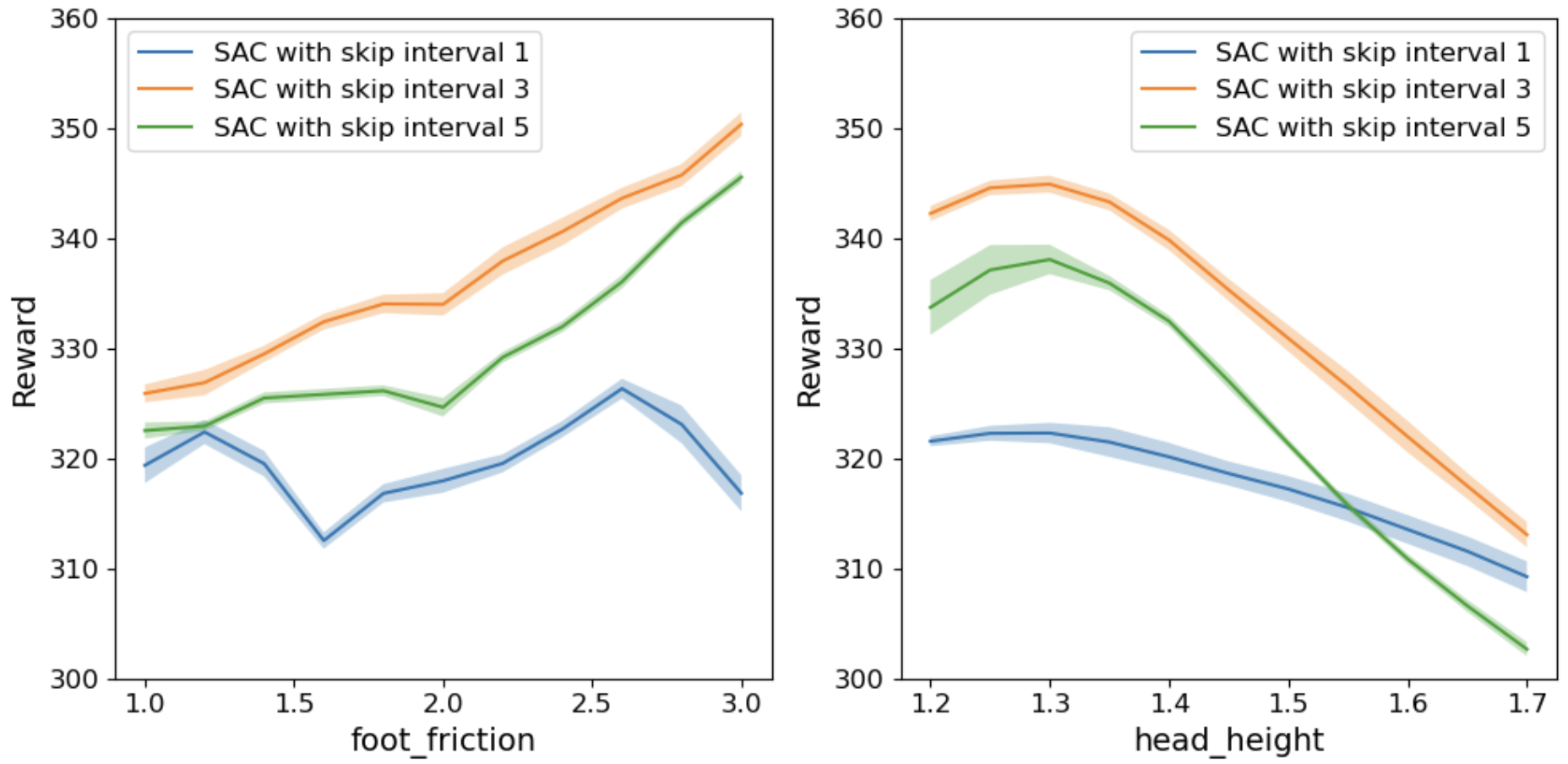}
        \caption{Impact of frame skip number.}
        \label{fig:frame_skip}
    \end{minipage}
\end{figure}

\begin{figure}[h]
    \centering
\includegraphics[scale = 0.32]{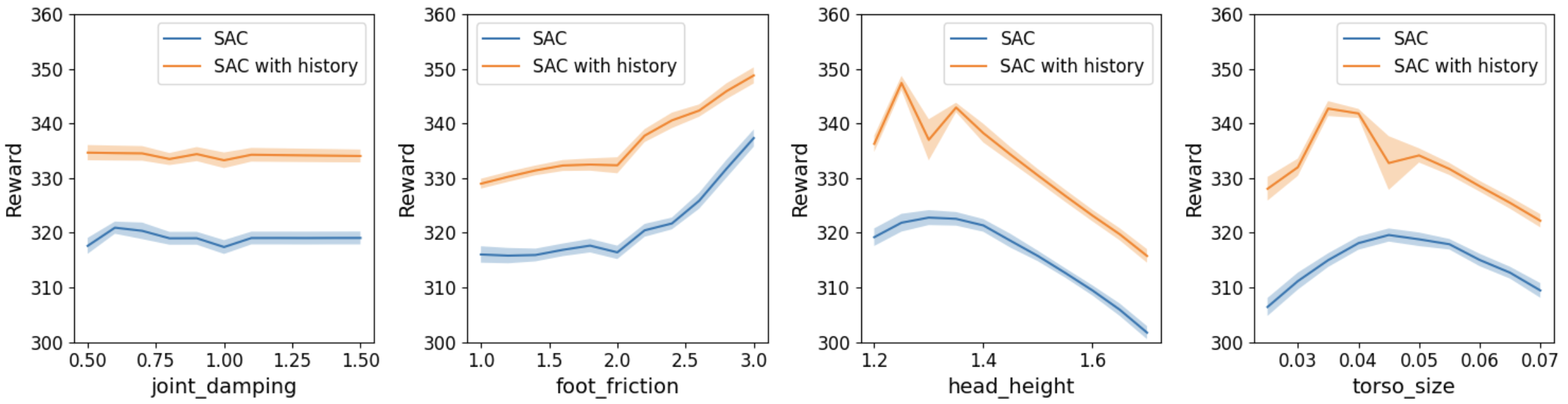}
\caption{Agents' reward in various test environments of hopper.}
\label{tab:hopper}
\end{figure}

\paragraph{Results on Other Environments}
Each algorithm was trained using three distinct random seeds in the half-cheetah and walker2d domain randomization (DR) environments. Consistent with previous experiments, we employed a frame stack number of $h=3$ and frame skip number of $m=3$. The comparative performance of our algorithm and the baseline algorithm, across various domain parameters, is presented in Figure~\ref{fig:walker2d}. The result clearly demonstrates that our algorithm consistently outperforms the baseline in all evaluated test environments.

\begin{figure}[h]
    \centering
    \includegraphics[scale=0.31]{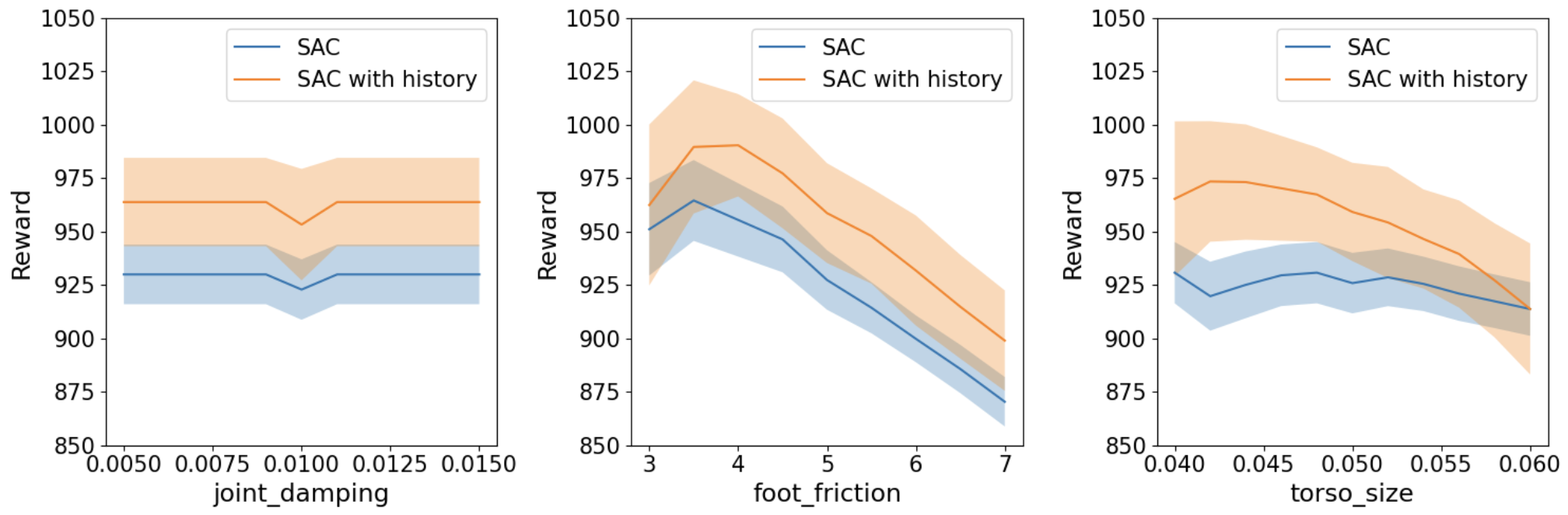}
    \vspace{1mm}
    \includegraphics[scale=0.26]{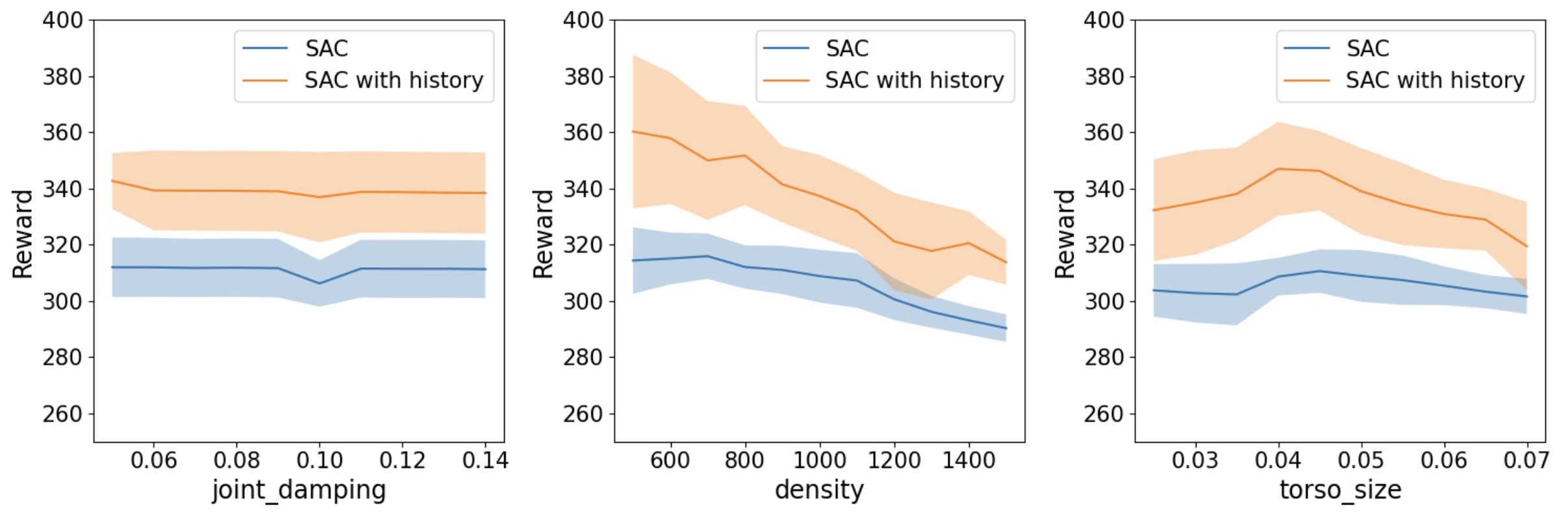}
    \caption{Performance in half-cheetah(Top) and walker2d(Bottom).}
    \label{fig:walker2d}
\end{figure}

\section{Conclusion, Limitations and Future Directions}
\label{sec:conclusion}

In this paper, we propose a two-level online controller for continuous-time linear systems with adversarial disturbances, aiming to achieve sublinear regret. This approach is grounded in our examination of agent training in domain randomization environments from an online control perspective. At the higher level, our controller employs the Online Convex Optimization (OCO) with memory framework to update policies at a low frequency, thus reducing regret. The lower level uses the DAC policy to align the system's actual state more closely with the idealized setting.

In our empirical evaluation, applying our algorithm's core principles to the SAC (Soft Actor-Critic) algorithm led to significantly improved results in multiple reinforcement learning tasks within domain randomization environments. This highlights the adaptability and effectiveness of our approach in practical scenarios.

 It is important to note that our theoretical analysis depends on the known dynamics of the system and the assumption of convex costs. This reliance could represent a limitation to our method, as it may not adequately address scenarios where these conditions do not hold or where system dynamics are incompletely understood. For future research, there are several promising directions in online non-stochastic control of continuous-time systems. These include extending our methods to systems with unknown dynamics, exploring the impact of assuming strong convexity in cost functions, and shifting the focus from regret to the competitive ratio. Further research can also explore how to utilize historical information more effectively to enhance agent training in domain randomization environments. This might involve employing time series analysis instead of simply incorporating parameters into neural network training.

\bibliography{ref}

\clearpage
\appendix

In the appendix we define $n$ as the smallest integer greater than or equal to $\frac{T}{h}$, and we use the shorthand $c_{ih}$, $x_{ih}$, $u_{ih}$, and $w_{ih}$ as $c_i$, $x_i$, $u_i$, and $w_i$, respectively. First we provide the proof of our main theorem there.
\section{Proof of Theorem \ref{thm:main}}
\main*

\begin{proof}
    We denote $u^\ast_t = K^\ast x^\ast_t$ as the optimal state and action that follows the policy specified by $K^\ast$, where $K^\ast = \arg\max_{K \in \mathcal{K}} J_T(K)$. 

We then discretize and decompose the regret as follows:
\begin{align*}
    J_T(\mathcal{A})-\min_{K \in \mathcal{K}} J_T(K) &=\int_0^T c_t(x_t, u_t) dt - \int_0^T c_t(x^{*}_t, u^{*}_t) dt \\
    &= \sum_{i=0}^{n-1} \int_{ih}^{(i+1)h} c_t(x_t, u_t) dt - \sum_{i=0}^{n-1} \int_{ih}^{(i+1)h} c_t(x^\ast_t, u^\ast_t) dt \\
    &= h \left(\sum_{i=0}^{n-1} c_{i}(x_{i}, u_{i}) - \sum_{i=0}^{n-1} c_{i}(x^*_{i}, u^*_{i})\right) + R_0 \,,
\end{align*}
where $R_0$ represents the discretization error.

We define $p$ as the smallest integer greater than or equal to $\frac{n}{m}$, then the first term can be further decomposed as
\begin{align*}
    & \sum_{i=0}^{n-1} c_{i}(x_{i}, u_{i}) - \sum_{i=0}^{n-1} c_{i}(x^*_{i}, u^*_{i}) \\ 
    = & \sum_{i=0}^{p - 1} \sum_{j=im}^{(i+1)m - 1} c_{i}(x_{i}, u_{i}) - \sum_{i=0}^{p - 1} \sum_{j=im}^{(i+1)m - 1} c_{i}(x^*_{i}, u^*_{i}) \\ 
    = & \sum_{i=0}^{p - 1} \left(\sum_{j=im}^{(i+1)m - 1} c_{i}(x_{i}, u_{i}) - \sum_{j=im}^{(i+1)m - 1} c_{i}(y_{i}, v_{i})\right) + \sum_{i=0}^{p - 1} \sum_{j=im}^{(i+1)m - 1}c_{i}(y_{i}, v_{i}) - \sum_{i=0}^{p - 1} \sum_{j=im}^{(i+1)m - 1} c_{i}(x^*_{i}, u^*_{i}) \\ 
    = & \sum_{i=0}^{p - 1} \left(\sum_{j=im}^{(i+1)m - 1} c_{i}(x_{i}, u_{i}) - f_i(\Tilde{M}_{i-H}, \ldots , \Tilde{M_i})\right) + \sum_{i=0}^{p - 1} f_i(\Tilde{M}_{i-H}, \ldots , \Tilde{M_i}) \\
    &- \min_{M \in \mathcal{M}}\sum_{i=0}^{p - 1} f_i(M, \ldots , M) 
    + \min_{M \in \mathcal{M}}\sum_{i=0}^{p - 1} f_i(M, \ldots , M)- \sum_{i=0}^{p - 1} \sum_{j=im}^{(i+1)m - 1} c_{i}(x^*_{i}, u^*_{i}) \,,
\end{align*}
where the last equality is by the definition of the idealized cost function. 

Let us denote 
\begin{align*}
    R_1 = \ &  \sum_{i=0}^{p - 1} \left(\sum_{j=im}^{(i+1)m - 1} c_{i}(x_{i}, u_{i}) - f_i(\Tilde{M}_{i-H}, \ldots , \Tilde{M_i})\right) \,, \\
    R_2 = \ &  \sum_{i=0}^{p - 1} f_i(\Tilde{M}_{i-H}, \ldots , \Tilde{M_i}) - \min_{M \in \mathcal{M}}\sum_{i=0}^{p - 1} f_i(M, \ldots , M) \,, \\
    R_3 = \ & \min_{M \in \mathcal{M}}\sum_{i=0}^{p - 1} f_i(M, \ldots , M)- \sum_{i=0}^{p - 1} \sum_{j=im}^{(i+1)m - 1} c_{i}(x^*_{i}, u^*_{i}) \,.
\end{align*}

Then we have the regret decomposition as
$$\mathrm{Regret}(T) = h(R_1 + R_2 + R_3) + O(hT) \,.$$

We then separately upper bound each of the four terms.

The term $R_0$ represents the error caused by discretization, which decreases as the number of sampling points increases and the sampling distance $h$ decreases. This is because more sampling points make our approximation of the continuous system more accurate. Using Lemma \ref{lem:dis}, we get the following upper bound: $R_0 \le O(hT)$.

The term $R_1$ represents the difference between the actual cost and the approximate cost. For a fixed $h$, this error decreases as the number of sample points looked ahead $m$ increases, while it increases as the sampling distance $h$ decreases. This is because the closer adjacent points are, the slower the convergence after approximation. By Lemma \ref{lem:app} we can bound it as $R_1 \le O(n(1-h\gamma)^{Hm})$.

The term $R_2$ is incurred due to the regret of the OCO with memory algorithm. Note that this term is determined by learning rate $\eta$ and the policy update frequency $m$. Choosing suitable parameters and using Lemma \ref{lem:oco}, we can obtain the following upper bound: $R_2 \le O(\sqrt{n/h})$.

The term $R_3$ represents the difference between the ideal optimal cost and the actual optimal cost. Since the accuracy of the DAC policy approximation of the optimal policy depends on its degree of freedom $l$, a higher degree of freedom leads to a more accurate approximation of the optimal policy. We use Lemma \ref{lem:opt} and choose $l = Hm$ to bound this error: $R_3 \le O(n(1-h\gamma)^{Hm})$.

By summing up these four terms and taking $m = \Theta(\frac{1}{h})$, we get:
\[\text{Regret}(T) \le O(nh(1-h\gamma)^{\frac{H}{h}}) + O(\sqrt{nh}) + O(hT) \,.\]
Finally, we choose $h = \Theta\left(\frac{1}{\sqrt{T}}\right)$, $m = \Theta\left(\frac{1}{h}\right)$, $H = \Theta(log(T))$, the regret is bounded by \[\text{Regret}(T) \le O(\sqrt{T}\log(T)) \,.\]

\end{proof}

\section{Key Lemmas}
In this section, we will primarily discuss the rationale behind the proof of our key lemmas. 
First, we need to prove all the states and actions are bounded. 
\begin{restatable}{lemma}{dif}
\label{lem:dif}
    Under Assumption \ref{asmp:bounded_dynamic_matrix} and \ref{asmp:bounded_cost}, choosing arbitrary $h$ in the interval $[0, h_0]$ where $h_0$ is a constant only depends on the parameters in the assumption, we have for any $t$ and policy $M_i$, $\|x_t\|, \|y_t\|, \|u_t\|, \|v_t\| \le D$, $\|\dot{x}_t\| \le D$, $\|x_t - y_t\|, \|u_t - v_t\| \le \kappa^2(1+\kappa)(1-h\gamma)^{Hm+1} D$. In particular, taking all the $M_t = 0$ and $K = K^*$, we can also obtain the inequality of the optimal solution: $\|x^*_t\|, \|u^*_t\| \le D$.
\end{restatable}

The proof of this Lemma mainly use the Gronwall inequality and the induction method.
Then we analyze the discretization error of the system.
\begin{restatable}{lemma}{dis}
\label{lem:dis}
    Under Assumption \ref{asmp:bounded_cost}, Algorithm \ref{alg} attains the following bound of $R_0$:
\begin{align*}
    R_0 = \sum_{i=0}^{n-1} \int_{ih}^{(i+1)h} (c_t(x_t, u_t) - c_t(x^\ast_t, u^\ast_t)) dt - h \sum_{i=0}^{n-1} \left(c_{i}(x_{i}, u_{i}) -  c_{i}(x^*_{i}, u^*_{i})\right)
    \le (G+L)D^2 hT \,.
\end{align*}
\end{restatable}

This lemma indicates that the discretization error is directly proportional to the sample distance $h$. In other words, increasing the number of sampling points leads to more accurate estimation of system.

Then we analysis the difference between ideal cost and actual cost.
The following lemma describes the upper bound of the error by approximating the ideal state and action:

\begin{restatable}{lemma}{app}
\label{lem:app}
    Under Assumption \ref{asmp:bounded_dynamic_matrix} and \ref{asmp:bounded_cost}, Algorithm \ref{alg} attains the following bound of $R_1$:
\begin{align*}
R_1 = \ & \sum_{i=0}^{p - 1} \left(\sum_{j=im}^{(i+1)m - 1} c_{i}(x_{i}, u_{i}) - f_i\left(\Tilde{M}_{i-H}, \ldots , \Tilde{M_i}\right)\right) 
\le nG D^2 \kappa^2(1+\kappa)(1-h\gamma)^{Hm+1} \,.
\end{align*}
\end{restatable}

From this lemma, it is evident that for a fixed sample distance $h$, the error diminishes as the number of sample points looked ahead $m$ increases. However, as the sampling distance $h$ decreases, the convergence rate of this term becomes slower. Therefore, it is not possible to select an arbitrarily small value for $h$ in order to minimize the discretization error $R_0$.

We need to demonstrate that the discrepancy between $x_t$ and $y_t$, as well as $u_t$ and $v_t$, is sufficiently small, given assumption \ref{asmp:bounded_dynamic_matrix}. This can be proven by analyzing the state evolution under the DAC policy.

By utilizing Assumption \ref{asmp:bounded_cost} and Lemma \ref{lem:dif}, we can deduce the following inequality:
\begin{align*}
\left|c_t\left(x_t, u_t\right)-c_t\left(y_t, v_t\right)\right| \le \ & \left|c_t\left(x_t, u_t\right)-c_t\left(y_t, u_t\right)\right| + \left|c_t\left(y_t, u_t\right)-c_t\left(y_t, v_t\right)\right| \\
\le \ & GD \|x_t - y_t\| +  GD \|u_t - v_t\| \,.
\end{align*}

Summing over all the terms and use Lemma \ref{lem:dif}, we can derive an upper bound for $R_1$.

Next, we analyze the regret of Online Convex Optimization (OCO) with a memory term. To analyze OCO with a memory term, we provide an overview of the framework established by \cite{anava2015online} in online convex optimization. The framework considers a scenario where, at each time step $t$, an online player selects a point $x_t$ from a set $\mathcal{K}\subset \mathbb{R}^d$. At each time step, a loss function $f_t: \mathcal{K}^{H+1} \rightarrow \mathbb{R}$ is revealed, and the player incurs a loss of $f_t\left(x_{t-H}, \ldots, x_t\right)$. The objective is to minimize the policy regret, which is defined as
$$
\mathrm{PolicyRegret }=\sum_{t=H}^T f_t\left(x_{t-H}, \ldots, x_t\right)-\min _{x \in \mathcal{K}} \sum_{t=H}^T f_t(x, \ldots, x) \,.
$$

In this setup, the first term corresponds to the DAC policy we choose, while the second term is used to approximate the optimal strongly stable linear policy.

\begin{restatable}{lemma}{oco}\label{lem:oco}
Under Assumption \ref{asmp:bounded_dynamic_matrix} and \ref{asmp:bounded_cost}, choosing $m = \frac{C}{h}$ and $\eta = \Theta(\frac{m}{Th})$, Algorithm \ref{alg} attains the following bound of $R_2$:
\begin{align*}
R_2 &= \sum_{i=0}^{p - 1} f_i(\Tilde{M}_{i-H}, \ldots , \Tilde{M_i}) - \min_{M \in \mathcal{M}}\sum_{i=0}^{p - 1} f_i(M, \ldots , M) \\
&\le 
    \frac{4a}{\gamma} \sqrt{\frac{GDC^2 \kappa^2(\kappa+1) W_0 \kappa_B}{\gamma}(\frac{GDC \kappa^2(\kappa+1) W_0 \kappa_B}{\gamma}+ C^2 \kappa^3 \kappa_B W_0 H^2) \frac{n}{h}} \,.
\end{align*}
\end{restatable}
To analyze this term, we can transform the problem into an online convex optimization with memory and utilize existing results presented by \cite{anava2015online} for it. By applying their results, we can derive the following bound:
$$
\sum_{t=H}^T f_t\left(x_{t-H}, \ldots, x_t\right)-\min _{x \in \mathcal{K}} \sum_{t=H}^T f_t(x, \ldots, x) \leq O\left(D \sqrt{G_f\left(G_f+L H^2\right) T}\right)  \,.
$$
Taking into account the bounds on the diameter, Lipschitz constant, and the gradient, we can ultimately derive an upper bound for $R_2$.

Lastly, we aim to establish a bound on the approximation error between the optimal DAC policy and the unknown optimal linear policy.
\begin{restatable}{lemma}{opt}
    \label{lem:opt}
    Under Assumption \ref{asmp:bounded_dynamic_matrix} and \ref{asmp:bounded_cost}, Algorithm \ref{alg} attains the following bound of $R_3$: 
\begin{align*}
    R_3 &= \min_{M \in \mathcal{M}}\sum_{i=0}^{p - 1} f_i(M, ... , M)- \sum_{i=0}^{p - 1} \sum_{j=im}^{(i+1)m - 1} c_{i}(x^*_{i}, u^*_{i}) \le  3n(1-h\gamma)^{Hm}GDW_0\kappa^3a(lh \kappa_B + 1) \,.
\end{align*}
\end{restatable}

The intuition behind this lemma is that the evolution of states leads to an approximation of the optimal linear policy in hindsight, where $u_t^* = -K^* x_t$ if we choose $M^* = \{M^i\}$, where $M^i = (K- K^*)(I + h(A-B K^*))^i$. Although the optimal policy $K^*$ is unknown, such an upper bound is attainable because the left-hand side represents the minimum of $M \in \mathcal{M}$.

\section{The evolution of the state}

In this section we will prove that using the DAC policy, the states and actions are uniformly bounded. The difference between ideal and actual states and the difference between ideal and actual action is very small.

We begin with expressions of the state evolution using DAC policy:

\begin{lemma}\label{lem:evo}
We have the evolution of the state and action:
\begin{align*}
    x_{t+1} = \ & Q^{l+1}_h x_{t-l} + h\sum_{i=0}^{2l} \Psi_{t, i} \hat{w}_{t-i} \,,\\
    y_{t+1} = \ & h\sum_{i=0}^{2Hm} \Psi_{t, i} \hat{w}_{t-i}\,, \\
    v_t = \ & -K y_t+h\sum_{j=1}^{Hm} M_t^{j} \hat{w}_{t-j} \,.
\end{align*}
where $\Psi_{t, i}$ represent the coefficients of $\hat{w}_{t-i}$:
    $$
\Psi_{t, i}=Q^i_h \mathbf{1}_{i \leq l}+h\sum_{j=0}^{l} Q^j_h B M_{t-j}^{i-j} \mathbf{1}_{i-j \in[1, l]} \,.
$$
\end{lemma}

\begin{proof}
Define $Q_h= I + h(A-B K)$. Using the Taylor expansion of $x_t$ and denoting $r_t$ as the second-order residue term, we have 
\begin{align*}
    x_{t+1} = x_t + h \dot{x}_t + h^2 r_t = x_t + h(Ax_t + Bu_t + w_t) + h^2 r_t \,.
\end{align*}

Then we calculate the difference between $w_i$ and $\hat{w}_i$:
\begin{align*}
    \hat{w}_t - w_t = \frac{x_{t+1} - x_t - h(Ax_t + Bu_t + w_t)}{h}  = h r_t\,.
\end{align*}

Using the definition of DAC policy and the difference between disturbance, we have
\begin{align*}
    x_{t+1} &= x_t + h\left(Ax_t + B\left(-K x_t+h\sum_{i=1}^{l} M_t^i \hat{w}_{t-i}\right) + \hat{w}_t - hr_t\right) + h^2 r_t \\
    &= (I + h(A - BK)) x_t + h\left(B h\sum_{i=1}^{l} M_t^i  \hat{w}_{t-i} + \hat{w}_t\right)  \\
    &= Q_h x_t + h\left(Bh\sum_{i=1}^{l} M_t^i  \hat{w}_{t-i} + \hat{w_t}\right)  \\
    &= Q^2_h x_{t-1} + h\left(Q_h\left(Bh\sum_{i=1}^{l} M_{t-1}^i  \hat{w}_{t-1-i} + \hat{w}_{t-1}\right)\right) +h\left(Bh\sum_{i=1}^{l} M_{t}^i  \hat{w}_{t-i} + \hat{w_t}\right)  \\
    &= Q^{l+1}_h x_{t-l} + h\sum_{i=0}^{2l} \Psi_{t, i} \hat{w}_{t-i} \,,
\end{align*}
where the last equality is by recursion and $\Psi_{t, i}$ represent the coefficients of $\hat{w}_{t-i}$.

Then we calculate the coefficients of $w_{t-i}$ and get the following result:
$$
\Psi_{t, i}=Q^i_h \mathbf{1}_{i \leq l}+h\sum_{j=0}^{l} Q^j_h B M_{t-j}^{i-j} \mathbf{1}_{i-j \in[1, l]} \,.
$$

By the ideal definition of $y_{t+1}$ and $v_t$(only consider the effect of the past $Hm$ steps while planning, assume $x_{t-Hm} = 0$), taking $l = Hm$ we have
\begin{align*}
y_{t+1} = \ & h\sum_{i=0}^{2Hm} \Psi_{t, i} \hat{w}_{t-i}, \\
v_t = \ & -K y_t+h\sum_{j=1}^{Hm} M_t^{j} \hat{w}_{t-j} \,. \tag*\qedhere
\end{align*}
\end{proof}

Then we prove the norm of the transition matrix is bounded.
\begin{lemma}\label{lem:transition}
    We have the following bound of the transition matrix:
    $$
\left\|\Psi_{t, i}\right\| \leq a(lh \kappa_B + 1) \kappa^2(1-h\gamma)^{i-1} \,.
$$
\end{lemma}

\begin{proof}
By the definition of strongly stable policy, we know 
\begin{align}\label{equ:strong}
    \|Q^i_h\| = \|(P L_h P^{-1})^i \|= \|P (L_h)^i P^{-1}\| \le \|P\| \|L_h\|^i \|P^{-1}\| \le a\kappa^2 (1-h\gamma)^i \,.
\end{align}

By the definition of $\Psi_{t, i}$, we have
\begin{align*}
\left\|\Psi_{t, i}\right\|
&= \left\|Q^i_h \mathbf{1}_{i \leq l}+h\sum_{j=0}^{l} Q^j_h B M_{t-j}^{i-j} \mathbf{1}_{i-j \in[1, l]} \right\| \\
& \le \kappa^2(1-h\gamma)^i +ah\sum_{j=1}^{l} \kappa_B \kappa^2 (1-h\gamma)^j (1-h\gamma)^{i-j-1}\\
& \leq \kappa^2(1-h\gamma)^i +alh \kappa_B \kappa^2 (1-h\gamma)^{i-1} \le a(lh \kappa_B + 1) \kappa^2(1-h\gamma)^{i-1}\,,
\end{align*}
where the first inequality is due to equation \ref{equ:strong}, assumption \ref{asmp:bounded_dynamic_matrix} and the condition of $\left\|M_t^{i}\right\| \leq a(1-h\gamma)^{i-1}$.
\end{proof}

After that, we can uniformly bound the state $x_t$ and its first and second-order derivative.
\begin{lemma}\label{lem:bound}
For any $t \in [0,T]$, choosing arbitrary $h$ in the interval $[0, h_0]$ where $h_0$ is a constant only depends on the parameters in the assumption, we have $\|x_t\| \le D_1$, $\|\dot{x}_t\| \le D_2$, $\|\ddot{x}_t\| \le D_3$ and the estimatation of disturbance is bounded by $\|\hat{w}_t\| \le W_0$. Moreover, $D_1$, $D_2$, $D_3$ are only depend on the parameters in the assumption.
\end{lemma}

\begin{proof}
    We prove this lemma by induction. When $t = 0$, it is clear that $x_0$ satisfies this condition. Suppose $x_t \le D_1$, $\dot{x}_t \le D_2$, $\ddot{x}_t \le D_3$, $\hat{w}_t \le W_0$ for any $t \le t_0$, where $t_0 = kh$ is the $k$-th discretization point. Then for $t \in [t_0, t_0 + h]$, we first prove that $\dot{x}_t \le D_2$, $\ddot{x}_t \le D_3$.

    By Assumption \ref{asmp:bounded_dynamic_matrix} and our definition of $u_t$, we know that for any $t \in [t_0, t_0 + h]$. Thus, we have
    \begin{align*}
        \|\dot{x}_t\| 
        &= \| Ax_t + Bu_t + w_t \| \\
        & = \| Ax_t + B(-Kx_{t_0} + h \sum_{i=1}^l M^i_{k} \hat{w}_{k-i} ) + w_t \| \\
        & \le \kappa_A \|x_t\| + \kappa_B \kappa\|x_{t_0}\| + h \sum_{i=1}^l (1-h\gamma)^{i-1} W_0 + W \\
        & \le \kappa_A \|x_t\| + \kappa_B \kappa D_1 + \frac{W_0}{\gamma} + W\,,
    \end{align*}
    where the first inequality is by the induction hypothesis $\hat{w}_t \le W_0$ for any $t \le t_0$ and $M^i_k \le (1-h \gamma)^{i-1}$, the second inequality is by the induction hypothesis $x_t \le D_1$ for any $t \le t_0$.

    For any $t \in [t_0, t_0 + h]$, because we choose the fixed policy $u_t \equiv u_{t_0}$, so we have $\dot{u}_t = 0$ and 
    \begin{align*}
        \| \ddot{x}_t \| &= \| A \dot{x}_t + B \dot{u}_t + \dot{w}_t \| =  \| A \dot{x}_t + \dot{w}_t \| \le \kappa_A \|\dot{x}_t\| + W \,.
    \end{align*}

    By the Newton-Leibniz formula, we have
    for any $\zeta \in [0,h]$, 
    \begin{align*}
        \dot{x}_{t_0 + \zeta} - \dot{x}_{t_0} = \int_{0}^{\zeta} \ddot{x}_{t_0 + \xi} d_{\xi} \,.
    \end{align*}

    Then we have 
    \begin{align*}
        \|\dot{x}_{t_0 + \zeta}\| &\le \|\dot{x}_{t_0}\| + \int_{0}^{\zeta} \|\ddot{x}_{t_0 + \xi}\| d_{\xi} \\
        & \le \|\dot{x}_{t_0}\| + \int_{0}^{\zeta} (\kappa_A\|\dot{x}_{t_0 + \xi}\| + W) d_{\xi} \\
        & = \|\dot{x}_{t_0}\|+ W\zeta + \kappa_A \int_{0}^{\zeta} \|\dot{x}_{t_0 + \xi}\| d_{\xi} \,.
    \end{align*}

    By Gronwall inequality, we have 
     \begin{align*}
        \|\dot{x}_{t_0 + \zeta}\| &\le \|\dot{x}_{t_0}\|+ W\zeta +  \int_{0}^{\zeta} (\|\dot{x}_{t_0}\|+W \xi)\exp(\kappa_A(\zeta - \xi)) d_{\xi} \,.
    \end{align*}

   Then we have 
    \begin{align*}
    \|\dot{x}_{t_0 + \zeta}\| &\le \|\dot{x}_{t_0}\|+ W\zeta +  \int_{0}^{\zeta} (\|\dot{x}_{t_0}\|+W \zeta)\exp(\kappa_A\zeta)) d_{\xi} \\
    &= (\|\dot{x}_{t_0}\|+ W\zeta)(1 + \zeta \exp(\kappa_A\zeta)) \\
    & \le \left(\kappa_A \|x_{t_0}\| + \kappa_B \kappa D_1 + \frac{W_0}{\gamma} + W + Wh \right)(1 + h\exp(\kappa_Ah)) \\
    & \le  \left((\kappa_A+ \kappa_B \kappa) D_1 + \frac{W_0}{\gamma} + W + Wh \right)(1 + h\exp(\kappa_Ah)) \\
    &\le \left((\kappa_A+ \kappa_B \kappa) D_1 + \frac{W_0}{\gamma} + 2W \right) (1+\exp(\kappa_A)) \,,
    \end{align*} 
    where the first inequality is by the relation $\xi \le \zeta$, the second inequality is by the relation $\zeta \le h$ and the bounding property of first-order derivative, the third inequality is by the induction hypothesis and the last inequality is due to $h \le 1$.

    By the relation $\|\ddot{x}_t \|\le \kappa_A\|\dot{x}_t\| + W$, we have 
    \begin{align*}
        \|\ddot{x}_{t_0 + \zeta}\| &\le \kappa_A D_2 + W \,.
    \end{align*}
    So we choose $D_3 = \kappa_A D_2 + W$.
    By the equation, we have 

    \begin{align*}
        \|\hat{w}_t - w_t\| &= \left\|\frac{x_{t+1} - x_t - h(Ax_t + Bu_t + w_t)}{h}\right\|  \\
        &= \left\|\frac{x_{t+1} - x_t - h\dot{x}_t}{h}\right\|  = \left\|\frac{\int_{0}^{h} (\dot{x}_{t+\xi} - \dot{x}_t)d\xi }{h}\right\|  = \left\|\frac{\int_{0}^{h} \int_{0}^{\xi}\ddot{x}_{t+\zeta} d\zeta d\xi}{h}\right\| 
        \\
        &\le \frac{\int_{0}^{h} \int_{0}^{\xi}\|\ddot{x}_{t+\zeta}\| d\zeta d\xi}{h} \\
        & \le hD_3 \,,
        \\
    \end{align*}
   where in the second line we use the Newton-Leibniz formula, the inequality is by the conclusion $\|\ddot{x}_t \|\le D_3$ which we have proved before.
    By Assumption \ref{asmp:bounded_dynamic_matrix}, we have
    $$
\|\hat{w}_t \|  \le W +  hD_3  \,.
$$ 

Choosing $D_3 = \kappa_A D_2 + W$, $W_0 = W + h D_3 = W + h(\kappa_A D_2 + W)$, we get

 \begin{align*}
    \|\dot{x}_{t_0 + \zeta}\| &\le  ((\kappa_A+ \kappa_B \kappa) D_1 + \frac{W_0}{\gamma} + 2W) (1+\exp(\kappa_A)) \\
    & \le  ((\kappa_A+ \kappa_B \kappa) D_1 + \frac{W + h(\kappa_A D_2 + W)}{\gamma} + 2W) (1+\exp(\kappa_A)) \\
    & \le D_2 \left(\frac{h\kappa_A}{\gamma} (1+\exp(\kappa_A))\right) + \left((\kappa_A+ \kappa_B \kappa) D_1 + \frac{(1+h+2\gamma)W}{\gamma}\right)(1+\exp(\kappa_A))) \,.
\end{align*} 

Using the notation 
\begin{align*}
    \beta_1 & = \frac{h\kappa_A}{\gamma} (1+\exp(\kappa_A)) \,, \\
    \beta_2 &= \left((\kappa_A+ \kappa_B \kappa) D_1 + \frac{2(1+\gamma)W}{\gamma}\right)(1+\exp(\kappa_A)) \,.
\end{align*}

When $h < \frac{\gamma}{2\kappa_A(1 + \exp(\kappa_A))}$, we have $\beta_1 < \frac{1}{2}$. Taking $D_2 = 2 \beta_2$ we get 
 \begin{align*}
    \|\dot{x}_{t_0 + \zeta}\| &\le \beta_1 D_2 + \beta_2 \le D_2 \,.
\end{align*} 

So we have proved that for any $t \in [t_0, t_0 + h]$, $\|\dot{x}_t\| \le D_2$, $\|\ddot{x}_t\| \le D_3$, $\|\hat{w}_t\| \le W_0$.

Then we choose suitable $D_1$ and prove that for any $t \in [t_0, t_0 + h]$, $\|x_t\| \le D_1$.

    Using Lemma \ref{lem:evo}, we have 
    \begin{align*}
    x_{t+1} &= h\sum_{i=0}^{t} \Psi_{t, i} \hat{w}_{t-i} \,. 
\end{align*}

By the induction hypothesis of bounded state and estimation noise in $[0, t_0]$ together with Lemma \ref{lem:transition}, we have 
\begin{align*}
    \|x_{t+1}\| &\le  h \sum_{i=0}^{t} (lh \kappa_B + 1) \kappa^2(1-h\gamma)^i ( W + h D_3) \\
    &\le \frac{(lh \kappa_B + 1) \kappa^2( W + h D_3)}{\gamma} \,.
\end{align*}

Then, by the Taylor expansion and the inequality $\dot{x}_t \le D_2$ , we have for any $\zeta \in [0,h]$,
\begin{align*}
   \|x_{t+1} - x_{t+\zeta}\| & = \|\int^h_\zeta \dot{x}_{t+\xi} d\xi \|  \le (h-\zeta) D_2 \le hD_2 \,.
\end{align*}

Therefore we have 
\begin{align*}
   \| x_{t+\zeta}\| & \le \|x_{t+1}\|  +  hD_2  \le \frac{(lh \kappa_B + 1) \kappa^2( W + h D_3)}{\gamma} + hD_2 \\
&=  \frac{(lh \kappa_B + 1) \kappa^2 W(1+h)}{\gamma} + hD_2 \left(\frac{(lh \kappa_B + 1) \kappa^2 \kappa_A}{\gamma}+1\right) \\
&\le  \frac{(l \kappa_B + 1) 2\kappa^2 W}{\gamma} + hD_2 \left(\frac{(l \kappa_B + 1) \kappa^2 \kappa_A}{\gamma}+1\right) \,.
\end{align*}

In the last inequality we use $h \le 1$.

By the relation $D_2 = \beta_2 / (1-\beta_1)$ and $\beta_1 \le \frac{1}{2}$, we know that 
$$D_2 \le 2 \left((\kappa_A+ \kappa_B \kappa) D_1 + \frac{2(1+\gamma)W}{\gamma}\right)(1+\exp(\kappa_A)).$$

Using the notation 
\begin{align*}
    \gamma_1 & =  2 h(\kappa_A+ \kappa_B \kappa)(1+\exp(\kappa_A)) \,, \\
    \gamma_2 &=  \frac{(l \kappa_B + 1) 2\kappa^2 W}{\gamma} + 4\frac{(1+\gamma)W}{\gamma}(1+\exp(\kappa_A))\left(\frac{(l \kappa_B + 1) \kappa^2 \kappa_A}{\gamma}+1\right) \,.
\end{align*}

We have $\| x_{t+\zeta}\|  \le
\gamma_1 D_1 +\gamma_2$.

From the equation of $\gamma_1$ we know 
that when
$h \le \frac{1}{4(\kappa_A+ \kappa_B \kappa)(1+\exp(\kappa_A))}$
we have $\gamma_1 \le \frac{1}{2}$. Then we choose $D_1 = 2\gamma_2$, we finally get 
\begin{align*}
   \| x_{t+\zeta}\| & \le
\gamma_1 D_1 +\gamma_2 \le D_1\,.
\end{align*}

Finally, set 
$$ h_0 = \min\left\{1, \frac{\gamma}{\kappa_A(1+\exp(\kappa_A))}, \frac{1}{4(\kappa_A+ \kappa_B \kappa)(1+\exp(\kappa_A))}\right\} \,,$$

By the relationship $D_1 = 2\gamma_2$, $D_2 = 2\beta_2$, $D_3 = \kappa_A D_2 + W$, $W_0 = W + hD_3$,  

we can verify the induction hypothesis. Moreover, we know that $D_1$, $D_2$, $D_3$ are not depend on $h$.
Therefore we have proved the claim. 

\end{proof}

The last step is then to bound the action and the approximation errors of states and actions.

\dif*

\begin{proof}
By Lemma \ref{lem:transition}, we have $$
\left\|\Psi_{t, i}\right\| \leq a(lh \kappa_B + 1) \kappa^2(1-h\gamma)^{i-1} \,.
$$

By Lemma \ref{lem:bound} we know that for any $h$ in $[0, h_0]$, where 
$$ h_0 = \min\left\{1, \frac{\gamma}{\kappa_A(1+\exp(\kappa_A))}, \frac{1}{4(\kappa_A+ \kappa_B \kappa)(1+\exp(\kappa_A))}\right\} \,,$$

we have
$\|x_t\| \le D_1$.

By Lemma \ref{lem:evo}, Lemma \ref{lem:transition} and Lemma \ref{lem:bound}, we have 
\begin{align*}
    \|y_{t+1} \| &=  \|h\sum_{i=0}^{2Hm} \Psi_{t, i} \hat{w}_{t-i} \| \\
    &\le h W_0 \sum_{i=0}^{2Hm} a(lh \kappa_B + 1) \kappa^2 (1-h\gamma)^{i-1} \\
    &\le \frac{ aW_0(lh \kappa_B + 1) \kappa^2}{\gamma} = \Tilde{D}_1\,. 
\end{align*}

Via the definition of $x_t, y_t$, we have
$$
\left\|x_t-y_t\right\| \leq \kappa^2(1-h\gamma)^{Hm+1}\left\|x_{t-Hm}\right\| \leq \kappa^2(1-h\gamma)^{Hm+1} D_1\,.
$$

For the actions 
\begin{align*}
    u_t= \ & -K x_t+h\sum_{i=1}^{Hm} M_t^i \hat{w}_{t-i} \,, \\
    v_t= \ & -K y_t+h\sum_{i=1}^{Hm} M_t^i \hat{w}_{t-i} \,,
\end{align*}
we can derive the bound
\begin{align*}
    &\left\|u_t\right\| \leq \  \left\|K x_t\right\|+h\sum_{i=1}^{Hm}\left\|M_t^{i} \hat{w}_{t-i}\right\| \leq \kappa\left\|x_t\right\|+W_0 h \sum_{i=1}^{Hm} a(1-h\gamma)^{i-1} \leq \kappa D_1 + \frac{aW_0}{\gamma} \,, \\
    &\left\|v_t\right\| \leq \  \left\|K y_t\right\|+h\sum_{i=1}^{Hm}\left\|M_t^{i} \hat{w}_{t-i}\right\| \leq \kappa\left\|y_t\right\|+W_0 h \sum_{i=1}^{Hm} a(1-h\gamma)^{i-1} \leq \kappa \Tilde{D}_1 + \frac{aW_0}{\gamma} \,, \\
    &\left\|u_t-v_t\right\| \leq \  \|K\|\left\|x_t-y_t\right\|  \leq \kappa^3(1-h\gamma)^{Hm+1} D_1\,.
\end{align*}

By Lemma~\ref{lem:bound}, taking $D = \max\{D_1, D_2, \Tilde{D}_1, \kappa D_1 + \frac{W_0}{\gamma}, \kappa \Tilde{D}_1 + \frac{W_0}{\gamma}\}$, we get the following inequality: $\|x_t\|, \|y_t\|, \|u_t\|, \|v_t\| \le D$, $\|\dot{x}_t\| \le D$.

We also have $$\|x_t - y_t\|+ \|u_t - v_t\| \le \kappa^2(1-h\gamma)^{Hm+1} D_1 + \kappa^3(1-h\gamma)^{Hm+1} D_1 \le \kappa^2(1+\kappa)(1-h\gamma)^{Hm+1} D\,.$$

In particular, the optimal policy can be recognized as taking the DAC policy with all the $M_t$ equal to 0 and the fixed strongly stable policy $K = K^*$. So we also have $\|x^*_t\|, \|u^*_t\|\le D$.

\end{proof}

Now we have finished the analysis of evolution of the states. It will be helpful to prove the key lemmas in this paper.

\section{Proof of Lemma \ref{lem:dis}}
In this section we will prove the following lemma:
\dis*

\begin{proof}

By Assumption \ref{asmp:bounded_cost} and Lemma \ref{lem:dif}, since we use the unchanged policy $u_t$ in the interval $t \in [ih,(i+1)h]$, we have 
\begin{align*}
    |c_t(x_t, u_t) - c_{ih}(x_{ih}, u_{ih})| &\le |c_t(x_t,u_t) - c_t(x_{ih}, u_{ih})| + |c_t(x_{ih},u_{ih})-c_{ih}(x_{ih},u_{ih})| \\
&\le \max_{x}\left\|\nabla_x c_t(x, u)\right\|\|x_t-x_{ih}\| + L(t-ih)D^2 \\
&\le GD\|\int^t_{ih}\dot{x}_s ds\| + L(t-ih)D^2 \\
&\le (G+L)D^2(t-ih) \,.
\end{align*}

Therefore we have
\begin{align*}
    |&\sum_{i=0}^{n-1} \int_{ih}^{(i+1)h} c_t(x_t, u_t) dt 
    - h \sum_{i=0}^{n-1} c_{i}(x_{i}, u_{i})| \\
    =& |\sum_{i=0}^{n-1} \int_{ih}^{(i+1)h} (c_t(x_t, u_t)-  c_{ih}(x_{ih}, u_{ih}))dt |\\
    \le & (G+L)D^2 \sum_{i=0}^{n-1} \int_{ih}^{(i+1)h} (t-ih) dt =\frac{1}{2} (G+L)D^2nh^2 = \frac{1}{2}(G+L)D^2hT \,.
\end{align*}

A similar bound can easily be established by lemma \ref{lem:dif} about the optimal state and policy:
\begin{align*}
   |\sum_{i=0}^{n-1} \int_{ih}^{(i+1)h} c_t(x^\ast_t, u^\ast_t) dt - \sum_{i=0}^{n-1} c_{i}(x^*_{i}, u^*_{i})| \le \frac{1}{2}(G+L)D^2hT \,.
\end{align*}

Taking sum of the two terms we get $R_0 \le (G+L)D^2hT$.

\end{proof}

\section{Proof of Lemma \ref{lem:app}}

In this section we will prove the following lemma:
\app*

\begin{proof}
    
Using Lemma \ref{lem:dif} and Assumption \ref{asmp:bounded_cost}, have the approximation error between ideal cost and actual cost bounded as,
\begin{align*}
\left|c_t\left(x_t, u_t\right)-c_t\left(y_t, v_t\right)\right| & \le \left|c_t\left(x_t, u_t\right)-c_t\left(y_t, u_t\right)\right| + \left|c_t\left(y_t, u_t\right)-c_t\left(y_t, v_t\right)\right| \\
 & \le G D \|x_t - y_t\| + G D \|u_t - v_t\| \\
 &\le G D^2 \kappa^2(1+\kappa)(1-h\gamma)^{Hm+1} \,,
\end{align*}
where the first inequality is by triangle inequality, the second inequality is by Assumption \ref{asmp:bounded_cost}, Lemma \ref{lem:dif}, and the third inequality is by Lemma \ref{lem:dif}.

With this, we have 
\begin{align*}
    R_1 &= \sum_{i=0}^{p - 1} \left(\sum_{j=im}^{(i+1)m - 1} c_{i}(x_{i}, u_{i}) - f_i(\Tilde{M}_{i-H}, ... , \Tilde{M_i})\right)  \\
    &= \sum_{i=0}^{p - 1} \left(\sum_{j=im}^{(i+1)m - 1} c_{i}(x_{i}, u_{i}) - \sum_{j=im}^{(i+1)m - 1} c_{i}(y_{i}, v_{i})\right) \\ 
    & \le \sum_{i=0}^{p - 1} \sum_{j=im}^{(i+1)m - 1}G D^2 \kappa^2(1+\kappa)(1-h\gamma)^{Hm+1}    \le nG D^2 \kappa^2(1+\kappa)(1-h\gamma)^{Hm+1}   \,.
\end{align*}
\end{proof}

\section{Proof of Lemma \ref{lem:oco}}

Before we start the proof of Lemma \ref{lem:oco}, we first present an overview of the online convex optimization (OCO) with memory framework. Consider the setting where, for every $t$, an online player chooses some point $x_t \in$ $\mathcal{K} \subset \mathbb{R}^d$, a loss function $f_t: \mathcal{K}^{H+1} \mapsto \mathbb{R}$ is revealed, and the learner suffers a loss of $f_t\left(x_{t-H}, \ldots, x_t\right)$. We assume a certain coordinate-wise Lipschitz regularity on $f_t$ of the form such that, for any $j \in\{1, \ldots, H\}$, for any $x_1, \ldots, x_H, \Tilde{x}_j \in \mathcal{K}$
$$
\left|f_t\left(x_1, \ldots, x_j, \ldots, x_H\right)-f_t\left(x_1, \ldots, \Tilde{x}_j, \ldots, x_H\right)\right| \leq L\left\|x_j-\Tilde{x}_j\right\| \,.
$$
In addition, we define $\Tilde{f}_t(x)=f_t(x, \ldots, x)$, and we let
$$
G_f=\sup_{t \in\{1, \ldots, T\}, x \in \mathcal{K}}\left\|\nabla \Tilde{f}_t(x)\right\|, \quad D_f=\sup _{x, y \in \mathcal{K}}\|x-y\| \,.
$$

The resulting goal is to minimize the policy regret, which is defined as
$$
\mathrm{Regret}=\sum_{t=H}^T f_t\left(x_{t-H}, \ldots, x_t\right)-\min _{x \in \mathcal{K}} \sum_{t=H}^T f_t(x, \ldots, x) \,.
$$

\begin{algorithm}
    \caption{Online Gradient Descent with Memory (OGD-M)} \label{alg:oco}
    \begin{algorithmic}
    \STATE Input: Step size  $\eta$, functions $\left\{f_t\right\}_{t=m}^T$.
    \STATE Initialize $x_0, \ldots, x_{H-1} \in \mathcal{K}$  arbitrarily.
        \FOR{$t = H, \ldots, T$}
            \STATE Play $x_t$, suffer loss $ f_t\left(x_{t-H}, \ldots, x_t\right)$.
            \STATE  Set $x_{t+1}=\Pi_{\mathcal{K}}\left(x_t-\eta \nabla \Tilde{f}_t(x)\right)$.
        \ENDFOR 
    \end{algorithmic}
\end{algorithm}

To minimize this regret, a commonly used algorithm is the Online Gradient descent. By running the Algorithm \ref{alg:oco}, we may bound the policy regret by the following lemma:
\begin{lemma}\label{lem:OGD}

    Let $\left\{f_t\right\}_{t=1}^T$ be Lipschitz continuous loss functions with memory such that $\Tilde{f}_t$ are convex. Then by runnning algorithm \ref{alg:oco} itgenerates a sequence $\left\{x_t\right\}_{t=1}^T$ such that
$$
\sum_{t=H}^T f_t\left(x_{t-H}, \ldots, x_t\right)-\min _{x \in \mathcal{K}} \sum_{t=H}^T f_t(x, \ldots, x) \leq \frac{D^2_f}{\eta}+T G_f^2 \eta+L H^2 \eta G_f T \,.
$$
Furthermore, setting $\eta=\frac{D_f}{\sqrt{G_f\left(G_f+L H^2\right) T}}$ implies that
$$
\mathrm{ PolicyRegret } \leq 2D_f \sqrt{G_f\left(G_f+L H^2\right) T} \,.
$$
\end{lemma}

\begin{proof}
    By the standard OGD analysis \cite{DBLP:journals/corr/abs-1909-05207}, we know that
$$
\sum_{t=H}^T \Tilde{f}_t\left(x_t\right)-\min _{x \in \mathcal{K}} \sum_{t=H}^T \Tilde{f}_t(x) \leq \frac{D^2_f}{\eta}+T G^2 \eta .
$$
In addition, we know by the Lipschitz property, for any $t \geq H$, we have
$$
\begin{aligned}
\left|f_t\left(x_{t-H}, \ldots, x_t\right)-f_t\left(x_t, \ldots, x_t\right)\right| & \leq L \sum_{j=1}^H\left\|x_t-x_{t-j}\right\| \leq L \sum_{j=1}^H \sum_{l=1}^j\left\|x_{t-l+1}-x_{t-l}\right\| \\
& \leq L \sum_{j=1}^H \sum_{l=1}^j \eta\left\|\nabla \Tilde{f}_{t-l}\left(x_{t-l}\right)\right\| \leq L H^2 \eta G,
\end{aligned}
$$
and so we have that
$$
\left|\sum_{t=H}^T f_t\left(x_{t-H}, \ldots, x_t\right)-\sum_{t=H}^T f_t\left(x_t, \ldots, x_t\right)\right| \leq T L H^2 \eta G .
$$
It follows that
$$
\sum_{t=H}^T f_t\left(x_{t-H}, \ldots, x_t\right)-\min _{x \in \mathcal{K}} \sum_{t=H}^T f_t(x, \ldots, x) \leq \frac{D^2_f}{\eta}+T G_f^2 \eta+L H^2 \eta G_f T \,.
$$
\end{proof}

In this setup, the first term corresponds to the DAC policy we make, and the second term is used to approximate the optimal strongly stable linear policy. It is worth noting that the cost of OCO with memory depends on the update frequency $H$. Therefore, we propose a two-level online controller. The higher-level controller updates the policy with accumulated feedback at a low frequency to reduce the regret, whereas a lower-level controller provides high-frequency updates of the DAC policy to reduce the discretization error. In the following part, we define the update distance of the DAC policy as $l = Hm$, where $m$ is the ratio of frequency between the DAC policy update and OCO memory policy update.
Formally, we update the value of $M_t$ once every $m$ transitions, where $g_t$ represents a loss function.
$$M_{t+1} = \left\{
		\begin{aligned}
		& \Pi_{\mathcal{M}}\left(M_t-\eta \nabla g_t(M)\right) & &{\text{if }t \% m == 0} \\
		&M_t & & {\text{otherwise}} \,.
		\end{aligned}
		\right.$$

From now on, we denote $\Tilde{M}_t = M_{tm}$ for the convenience to remove the duplicate elements. By the definition of ideal cost, we know that it is a well-defined definition.

By Lemma \ref{lem:evo} we know that 
\begin{align*}
    & y_{t+1}  =h\sum_{i=0}^{2Hm} \Psi_{t, i} \hat{w}_{t-i}, \\
 & v_t  =-K y_t+h\sum_{j=1}^{Hm} M_t^{j} \hat{w}_{t-j} \,,
\end{align*}
where
$$
\Psi_{t, i}=Q^i_h \mathbf{1}_{i \leq l}+h\sum_{j=0}^{l} Q^j_h B M_{t-j}^{i-j} \mathbf{1}_{i-j \in[1, l]} \,.
$$

So we know that $y_t$ and $y_t$ are linear combination of $M_t$, therefore $$
f_i\left(\Tilde{M}_{i-H}, \ldots, \Tilde{M}_i\right)=\sum_{t = im}^{(i+1)m - 1}c_t\left(y_t\left(\Tilde{M}_{i-H}, \ldots, \Tilde{M}_{i}\right), v_t\left(\Tilde{M}_{i-H}, \ldots, \Tilde{M}_i\right)\right) .
$$ 
is convex in $M_t$. So we can use the OCO with memory structure to solve this problem. 

By Lemma \ref{lem:bound} we know that $y_t$ and $v_t$ are bounded by $D$. Then we need to calculate the diameter, Lipchitz constant, and gradient bound of this function $f_i$. In the following, we choose the DAC policy parameter $l = Hm$.

\begin{lemma}\label{lem:dia}
   (Bounding the diameter) 
We have $$D_f = \sup_{M_i, M_j \in \mathcal{M}} \|M_i - M_j\| \le \frac{2a}{h\gamma}$$.
\end{lemma}

\begin{proof}
    By the definition of $\mathcal{M}$, taking $l = Hm$ we know that 

    \begin{align*}
        \sup_{M_i, M_j \in \mathcal{M}} \|M_i - M_j\| &\le \sum_{k=1}^{Hm} \|M^k_i - M^k_j\| \\
        & \le \sum_{k=1}^{Hm} 2 a(1-h\gamma)^{k-1} \\
        & \le \frac{2a}{h\gamma} \,.
    \end{align*}
\end{proof}

\begin{lemma}\label{lem:lip}(Bounding the Lipschitz Constant) 
Consider two policy sequences $\left\{\Tilde{M}_{i-H} \ldots \Tilde{M}_{i-k} \ldots \Tilde{M}_i\right\}$ and $\left\{\Tilde{M}_{i-H} \ldots \hat{M}_{i-k} \ldots \Tilde{M}_i\right\}$ which differ in exactly one policy played at a time step $t-k$ for $k \in\{0, \ldots, H\}$. Then we have that
$$
\left|f_i\left(\Tilde{M}_{i-H} \ldots \Tilde{M}_{i-k} \ldots \Tilde{M}_i\right)-f_i\left(\Tilde{M}_{i-H} \ldots \hat{M}_{i-k} \ldots \Tilde{M}_i\right)\right| \leq C^2 \kappa^3 \kappa_B W_0 \sum_{j=0}^{Hm} \|\Tilde{M}^j_{i-k} - \hat{M}^j_{i-k}\|\,,
$$
where $C$ is a constant.
\end{lemma}

\begin{proof}
    By the definition we have 

\begin{align*}
\left\|y_t-\Tilde{y}_t\right\| & = \|h\sum_{i=0}^{2Hm} h\sum_{j=0}^{Hm} Q^j_h B (M_{t-j}^{i-j} - \Tilde{M}_{t-j}^{i-j}) \mathbf{1}_{i-j \in[1, Hm]}  \hat{w}_{t-i} \| \\
&\le h^2 \kappa^2 \kappa_B W_0 \sum_{i=0}^{2Hm}\sum_{j=0}^{Hm}\|M_{t-j}^{i-j} - \Tilde{M}_{t-j}^{i-j}\| \mathbf{1}_{i-j \in[1, Hm]} \\
&\le h^2 \kappa^2 \kappa_B W_0 m \sum_{j=0}^{Hm} \|\Tilde{M}^j_{i-k} - \hat{M}^j_{i-k}\| \\
&= hC \kappa^2 \kappa_B W_0 \sum_{j=0}^{Hm} \|\Tilde{M}^j_{i-k} - \hat{M}^j_{i-k}\| \,.
\end{align*}

Where the first inequality is by $\|Q^j_h\| \le \kappa^2 (1-h\gamma)^{j-1} \le \kappa^2$ and lemma \ref{lem:bound} of bounded estimation disturbance, the second inequality is by the fact that $M_{i-k}$ have taken $m$ times, the last equality is by $m = \frac{C}{h}$.
Furthermore, we have that
$$
\left\|v_t-\Tilde{v}_t\right\| =\|-K\left(y_t-\Tilde{y}_t\right) \| \le hC \kappa^3 \kappa_B W_0 \sum_{j=0}^{Hm}\left\|\Tilde{M}_{i-k}^{j}-\hat{M}_{i-k}^{j}\right\| \,.
$$
Therefore using Assumption \ref{asmp:bounded_cost}, Lemma \ref{lem:bound} and Lemma \ref{lem:dif} we immediately get that
$$
\left|f_i\left(\Tilde{M}_{i-H} \ldots \Tilde{M}_{i-k} \ldots \Tilde{M}_i\right)-f_i\left(\Tilde{M}_{i-H} \ldots \hat{M}_{i-k} \ldots \Tilde{M}_i\right)\right| \leq C^2 \kappa^3 \kappa_B W_0  \sum_{j=0}^{Hm} \|\Tilde{M}^j_{i-k} - \hat{M}^j_{i-k}\| \,.
$$

\end{proof}

\begin{lemma}\label{lem:gra}
    (Bounding the Gradient) We have the following bound for the gradient:
$$
\left\|\nabla_M f_t(M \ldots M)\right\|_F \leq \frac{GDC \kappa^2(\kappa+1) W_0 \kappa_B}{\gamma}
$$
\end{lemma}

\begin{proof}
Since $M$ is a matrix, the $\ell_2$ norm of the gradient $\nabla_M f_t$ corresponds to the Frobenius norm of the $\nabla_M f_t$ matrix. So it will be sufficient to derive an absolute value bound on $\nabla_{M_{p, q}^{[r]}} f_t(M, \ldots, M)$ for all $r, p, q$. To this end, we consider the following calculation. Using lemma \ref{lem:bound} we get that $y_t(M \ldots M), v_t(M \ldots M) \leq D$. Therefore, using Assumption \ref{asmp:bounded_cost} we have that
$$
\left|\nabla_{M_{p, q}^{[r]}} c_t(M \ldots M)\right| \leq G D\left(\left\|\frac{\partial y_t(M)}{\partial M_{p, q}^{[r]}}+\frac{\partial v_t(M \ldots M)}{\partial M_{p, q}^{[r]}}\right\|\right) .
$$

We now bound the quantities on the right-hand side:
$$
\begin{aligned}
\left\|\frac{\delta y_t(M \ldots M)}{\delta M_{p, q}^{[r]}}\right\| & =\left\|h\sum_{i=0}^{2 Hm} h\sum_{j=1}^{Hm}\left[\frac{\partial Q_h^j B M^{[i-j]}}{\partial M_{p, q}^{[r]}}\right] \hat{w}_{t-i} \mathbf{1}_{i-j \in[1, H]}\right\| \\
& \leq h^2\sum_{i=r}^{r+Hm}\left\|\left[\frac{\partial Q_h^{i-r} B M^{[r]}}{\partial M_{p, q}^{[r]}}\right] w_{t-i}\right\| \\
& \le h^2 \kappa^2 W_0 \kappa_B \frac{1}{h\gamma} = \frac{h \kappa^2 W_0 \kappa_B}{\gamma} \,.
\end{aligned}
$$
Similarly,
$$
\begin{aligned}
\left\|\frac{\partial v_t(M \ldots M)}{\partial M_{p, q}^{[r]}}\right\|  \leq \kappa\left\|\frac{\delta y_t(M \ldots M)}{\delta M_{p, q}^{[r]}}\right\|  \leq \kappa \frac{h \kappa^2 W_0 \kappa_B}{\gamma}  \le  \frac{h \kappa^3 W_0 \kappa_B}{\gamma}\,.
\end{aligned}
$$
Combining the above inequalities with $$
f_i\left(\Tilde{M}_{i-H}, \ldots, \Tilde{M}_i\right)=\sum_{t = im}^{(i+1)m - 1}c_t\left(y_t\left(\Tilde{M}_{i-H}, \ldots, \Tilde{M}_{i}\right), v_t\left(\Tilde{M}_{i-H}, \ldots, \Tilde{M}_i\right)\right) \,.
$$ gives the bound that 
\begin{align*}
    \left\|\nabla_M f_t(M \ldots M)\right\|_F \leq \frac{GDC \kappa^2(\kappa+1) W_0 \kappa_B}{\gamma}\,.
\end{align*}
\end{proof}

Finally we prove Lemma \ref{lem:oco}:

\oco*

\begin{proof}

By Lemma \ref{lem:OGD} we have 
$$
R_2 \le 2D_f \sqrt{G_f\left(G_f+L H^2\right) p}
$$

By Lemma \ref{lem:dia}, Lemma \ref{lem:lip}, and Lemma \ref{lem:gra} we have
\begin{align*}
    R_2 & \le 2D_f \sqrt{G_f\left(G_f+L H^2\right) p} \\
    & \le 2 \frac{2a}{h\gamma} \sqrt{\frac{GDC \kappa^2(\kappa+1) W_0 \kappa_B}{\gamma}(\frac{GDC \kappa^2(\kappa+1) W_0 \kappa_B}{\gamma}+ C^2 \kappa^3 \kappa_B W_0 H^2) \frac{n}{m}} \\
    & \le \frac{4a}{\gamma} \sqrt{\frac{GDC^2 \kappa^2(\kappa+1) W_0 \kappa_B}{\gamma}(\frac{GDC \kappa^2(\kappa+1) W_0 \kappa_B}{\gamma}+ C^2 \kappa^3 \kappa_B W_0 H^2) \frac{n}{h}} \,.
\end{align*}
\end{proof}

\section{Proof of Lemma \ref{lem:opt}}
In this section, we will prove the approximation value of DAC policy and optimal policy is sufficiently small. First, we introduce the following:

\begin{lemma}\label{lem:optimal}
    For any two $(\kappa, \gamma)$-strongly stable matrices $K^*, K$, there exists $M=\left(M^{1}, \ldots, M^{Hm}\right)$ where
$$
M^{i}=\left(K - K^*\right)\left(I + h(A-B K^*)\right)^{i-1} \,,
$$
such that
$$
c_{t}(x_{t}(M), u_{t}(M))- c_{t}(x^*_{t}, u^*_{t}) \leq GDW_0\kappa^3 a(lh \kappa_B + 1)(1-h\gamma)^{Hm}   \,.
$$
\end{lemma}

\begin{proof}
    Denote $Q_h(K) = I + h(A-BK)$, $Q_h(K^*) = I + h(A-BK^*)$. By Lemma \ref{lem:evo} we have 
$$
x^*_{t+1}=h\sum_{i=0}^tQ^i_h(K^*) \hat{w}_{t-i} \,.
$$
 Consider the following calculation for $i \le Hm$ and $
M^{i}=\left(K - K^*\right)\left(I + h(A-B K^*)\right)^{i-1}
$:
\begin{align*}
\Psi_{t, i}\left(M, \ldots, M\right) &=Q^i_h(K) +h\sum_{j=1}^{i} Q^{i-j}_h(K) B M^j\\
&=Q^i_h(K) +h\sum_{j=1}^{i} Q^{i-j}_h(K) B\left(K - K^*\right) Q^{j-1}_h(K^*) \\
&=Q^i_h(K) +\sum_{j=1}^{i} Q^{i-j}_h(K) (Q_h(K^*) - Q_h(K)) Q^{j-1}_h(K^*) \\
& =Q^i_h(K^*) \,,
\end{align*}
where the final equality follows as the sum telescopes. Therefore, we have that
$$
x_{t+1}(M)=h\sum_{i=0}^{Hm}Q^i_h(K^*)  \hat{w}_{t-i} + h\sum_{i=Hm+1}^{t} \Psi_{t, i} \hat{w}_{t-i} \,.
$$
Then we obtain that
$$
\left\|x_{t+1}(M)-x^*_{t+1}\right\| \leq hW_0\sum_{i=Hm+1}^t(\left\|\Psi_{t, i}\left(M_*\right)\right\| + \|Q^i_h(K^*)\|)  \,.
$$

Using Definition \ref{def:linear_policy} and Lemma \ref{lem:evo} we finally get 
\begin{align*}
    \left\|x_{t+1}(M)-x^*_{t+1}\right\| & \leq hW_0 (\sum_{i=Hm+1}^t ((lh \kappa_B + 1) a\kappa^2(1-h\gamma)^{i-1}) + \kappa^2 (1-h\gamma)^i) \\
    & \le W_0 (lh \kappa_B + 2) a\kappa^2(1-h\gamma)^{Hm}\,.
\end{align*}

We also have
$$
\begin{aligned}
\left\|u^*_t-u_t\left(M\right)\right\| & =\left\|-K^* x^*_t+K x_t\left(M\right)-h\sum_{i=0}^{Hm} M^i \hat{w}_{t-i} \right\| \\
& = \|(K-K^*)x^*_t + K(x_t(M) - x^*_t) -h\sum_{i=0}^{Hm} M^i \hat{w}_{t-i} \|\\
& = \|(K-K^*)h\sum_{i=0}^{t-1}Q^i_h(K^*) \hat{w}_{t-i} + K(x_t(M) - x^*_t) -h\sum_{i=0}^{Hm} M^i \hat{w}_{t-i} \|\\
& = \| K(x_t(M) - x^*_t) -h\sum_{i=Hm+1}^{t-1} (K-K^*)Q^{i-1}_h(K^*) \hat{w}_{t-i} \|\\
& = \| Kh\sum_{i=Hm+1}^{t-1} (\Psi_{t, i}-Q^{i-1}_h(K^*))\hat{w}_{t-i} -h\sum_{i=Hm+1}^{t-1} (K-K^*)Q^{i-1}_h(K^*) \hat{w}_{t-i} \|\\
& = \left\|h\sum_{i=Hm+1}^{t-1} K^*\left(Q^{i-1}_h(K^*)+\Psi_{t, i}\right) \hat{w}_{t-i}\right\| \\
& \leq  W_0 \kappa ((1-h\gamma)^{Hm} +  a(lh \kappa_B + 1) \kappa^2(1-h\gamma)^{Hm}) \\
&=   W_0 \kappa (a(lh \kappa_B + 1) \kappa^2 + 1)(1-h\gamma)^{Hm}) \,,
\end{aligned}
$$
where the inequality is by Definition \ref{def:linear_policy} and Lemma \ref{lem:transition}.

Finally, we have 
\begin{align*}
&\left|c_t\left(x_t(M), u_t(M)\right)-c_t\left(x^*_t, u^*_t\right)\right| \\
 \le& \left|c_t\left(x_t(M), u_t(M)\right)-c_t\left(x^*_t, u_t(M)\right)\right| + \left|c_t\left(x^*_t, u_t(M)\right)-c_t\left(x^*_t, u^*_t\right)\right| \\
 \le& GD |x_t(M) - x^*_t| + GD |u_t(M) - u^*_t| \\
 \le &  GDW_0\kappa^3 a(lh \kappa_B + 1)(1-h\gamma)^{Hm} \,,
\end{align*}
where the second inequality is by Assumption \ref{asmp:bounded_cost}.
\end{proof}

Then we can prove our main lemma:
\opt*
\begin{proof}
By choosing 
$$
M^{i}=\left(K - K^*\right)\left(I + h(A-B K^*)\right)^{i-1} \,.
$$
We know that  
$$
\|M^{i}\|=\|\left(K - K^*\right)\left(I + h(A-B K^*)\right)^{i-1}\| \le 2\kappa^3 (1-\gamma)^{i-1}  \,.
$$
Therefore choose $a = 2\kappa^3$ we have $M = \{M^i\}$ in the DAC policy update class $\mathcal{M}$.

   Then we have the analysis of the regret:
\begin{align*}
    R_3 &= \min_{M\in \mathcal{M}}\sum_{i=0}^{p - 1} f_i(M, ... , M)- \sum_{i=0}^{p - 1} \sum_{j=im}^{(i+1)m - 1} c_{i}(x^*_{i}, u^*_{i}) \\
    &\le \min_{M\in \mathcal{M}}\sum_{i=0}^{p - 1}\sum_{j=im}^{(i+1)m - 1} c_{i}(x_{i}(M), u_{i}(M))- \sum_{i=0}^{p - 1} \sum_{j=im}^{(i+1)m - 1} c_{i}(x^*_{i}, u^*_{i}) +  n\kappa^2(1+\kappa)(1-h\gamma)^{Hm+1} D\\
    & \le 3n(1-h\gamma)^{Hm}GDW_0\kappa^3 a(lh \kappa_B + 1)\,,
\end{align*}
where the first inequality is by Lemma \ref{lem:dif} and the second inequality is by Lemma \ref{lem:optimal}.

\end{proof}

\clearpage
\section*{NeurIPS Paper Checklist}

\begin{enumerate}

\item {\bf Claims}
    \item[] Question: Do the main claims made in the abstract and introduction accurately reflect the paper's contributions and scope?
    \item[] Answer: \answerYes{} 
    \item[] Justification: We clarify our contributions and basic problem setups in both abstract and introduction.
    \item[] Guidelines: 
    \begin{itemize}
        \item The answer NA means that the abstract and introduction do not include the claims made in the paper.
        \item The abstract and/or introduction should clearly state the claims made, including the contributions made in the paper and important assumptions and limitations. A No or NA answer to this question will not be perceived well by the reviewers. 
        \item The claims made should match theoretical and experimental results, and reflect how much the results can be expected to generalize to other settings. 
        \item It is fine to include aspirational goals as motivation as long as it is clear that these goals are not attained by the paper. 
    \end{itemize}

\item {\bf Limitations}
    \item[] Question: Does the paper discuss the limitations of the work performed by the authors?
    \item[] Answer: \answerYes{} 
    \item[] Justification: We discuss the limitation of our paper in Section~\ref{sec:conclusion}.
    \item[] Guidelines:
    \begin{itemize}
        \item The answer NA means that the paper has no limitation while the answer No means that the paper has limitations, but those are not discussed in the paper. 
        \item The authors are encouraged to create a separate "Limitations" section in their paper.
        \item The paper should point out any strong assumptions and how robust the results are to violations of these assumptions (e.g., independence assumptions, noiseless settings, model well-specification, asymptotic approximations only holding locally). The authors should reflect on how these assumptions might be violated in practice and what the implications would be.
        \item The authors should reflect on the scope of the claims made, e.g., if the approach was only tested on a few datasets or with a few runs. In general, empirical results often depend on implicit assumptions, which should be articulated.
        \item The authors should reflect on the factors that influence the performance of the approach. For example, a facial recognition algorithm may perform poorly when image resolution is low or images are taken in low lighting. Or a speech-to-text system might not be used reliably to provide closed captions for online lectures because it fails to handle technical jargon.
        \item The authors should discuss the computational efficiency of the proposed algorithms and how they scale with dataset size.
        \item If applicable, the authors should discuss possible limitations of their approach to address problems of privacy and fairness.
        \item While the authors might fear that complete honesty about limitations might be used by reviewers as grounds for rejection, a worse outcome might be that reviewers discover limitations that aren't acknowledged in the paper. The authors should use their best judgment and recognize that individual actions in favor of transparency play an important role in developing norms that preserve the integrity of the community. Reviewers will be specifically instructed to not penalize honesty concerning limitations.
    \end{itemize}

\item {\bf Theory Assumptions and Proofs}
    \item[] Question: For each theoretical result, does the paper provide the full set of assumptions and a complete (and correct) proof?
    \item[] Answer: \answerYes{} 
    \item[] Justification: We provide the full set of assumptions and a complete proof.
    \item[] Guidelines:
    \begin{itemize}
        \item The answer NA means that the paper does not include theoretical results. 
        \item All the theorems, formulas, and proofs in the paper should be numbered and cross-referenced.
        \item All assumptions should be clearly stated or referenced in the statement of any theorems.
        \item The proofs can either appear in the main paper or the supplemental material, but if they appear in the supplemental material, the authors are encouraged to provide a short proof sketch to provide intuition. 
        \item Inversely, any informal proof provided in the core of the paper should be complemented by formal proofs provided in appendix or supplemental material.
        \item Theorems and Lemmas that the proof relies upon should be properly referenced. 
    \end{itemize}

    \item {\bf Experimental Result Reproducibility}
    \item[] Question: Does the paper fully disclose all the information needed to reproduce the main experimental results of the paper to the extent that it affects the main claims and/or conclusions of the paper (regardless of whether the code and data are provided or not)?
    \item[] Answer: \answerYes{} 
    \item[] Justification: We disclose the experiment details in Section~\ref{sec:experiment}.
    \item[] Guidelines:
    \begin{itemize}
        \item The answer NA means that the paper does not include experiments.
        \item If the paper includes experiments, a No answer to this question will not be perceived well by the reviewers: Making the paper reproducible is important, regardless of whether the code and data are provided or not.
        \item If the contribution is a dataset and/or model, the authors should describe the steps taken to make their results reproducible or verifiable. 
        \item Depending on the contribution, reproducibility can be accomplished in various ways. For example, if the contribution is a novel architecture, describing the architecture fully might suffice, or if the contribution is a specific model and empirical evaluation, it may be necessary to either make it possible for others to replicate the model with the same dataset, or provide access to the model. In general. releasing code and data is often one good way to accomplish this, but reproducibility can also be provided via detailed instructions for how to replicate the results, access to a hosted model (e.g., in the case of a large language model), releasing of a model checkpoint, or other means that are appropriate to the research performed.
        \item While NeurIPS does not require releasing code, the conference does require all submissions to provide some reasonable avenue for reproducibility, which may depend on the nature of the contribution. For example
        \begin{enumerate}
            \item If the contribution is primarily a new algorithm, the paper should make it clear how to reproduce that algorithm.
            \item If the contribution is primarily a new model architecture, the paper should describe the architecture clearly and fully.
            \item If the contribution is a new model (e.g., a large language model), then there should either be a way to access this model for reproducing the results or a way to reproduce the model (e.g., with an open-source dataset or instructions for how to construct the dataset).
            \item We recognize that reproducibility may be tricky in some cases, in which case authors are welcome to describe the particular way they provide for reproducibility. In the case of closed-source models, it may be that access to the model is limited in some way (e.g., to registered users), but it should be possible for other researchers to have some path to reproducing or verifying the results.
        \end{enumerate}
    \end{itemize}

\item {\bf Open access to data and code}
    \item[] Question: Does the paper provide open access to the data and code, with sufficient instructions to faithfully reproduce the main experimental results, as described in supplemental material?
    \item[] Answer: \answerNo{} 
    \item[] Justification: Our code is very simple, just use the traditional SAC algorithm with one line implement. Our main contribution is the theoretical analysis.
    \item[] Guidelines:
    \begin{itemize}
        \item The answer NA means that paper does not include experiments requiring code.
        \item Please see the NeurIPS code and data submission guidelines (\url{https://nips.cc/public/guides/CodeSubmissionPolicy}) for more details.
        \item While we encourage the release of code and data, we understand that this might not be possible, so “No” is an acceptable answer. Papers cannot be rejected simply for not including code, unless this is central to the contribution (e.g., for a new open-source benchmark).
        \item The instructions should contain the exact command and environment needed to run to reproduce the results. See the NeurIPS code and data submission guidelines (\url{https://nips.cc/public/guides/CodeSubmissionPolicy}) for more details.
        \item The authors should provide instructions on data access and preparation, including how to access the raw data, preprocessed data, intermediate data, and generated data, etc.
        \item The authors should provide scripts to reproduce all experimental results for the new proposed method and baselines. If only a subset of experiments are reproducible, they should state which ones are omitted from the script and why.
        \item At submission time, to preserve anonymity, the authors should release anonymized versions (if applicable).
        \item Providing as much information as possible in supplemental material (appended to the paper) is recommended, but including URLs to data and code is permitted.
    \end{itemize}

\item {\bf Experimental Setting/Details}
    \item[] Question: Does the paper specify all the training and test details (e.g., data splits, hyperparameters, how they were chosen, type of optimizer, etc.) necessary to understand the results?
    \item[] Answer: \answerYes{} 
    \item[] Justification: We specify all the training and test details in~\ref{sec:experiment}.
    \item[] Guidelines:
    \begin{itemize}
        \item The answer NA means that the paper does not include experiments.
        \item The experimental setting should be presented in the core of the paper to a level of detail that is necessary to appreciate the results and make sense of them.
        \item The full details can be provided either with the code, in appendix, or as supplemental material.
    \end{itemize}

\item {\bf Experiment Statistical Significance}
    \item[] Question: Does the paper report error bars suitably and correctly defined or other appropriate information about the statistical significance of the experiments?
    \item[] Answer: \answerYes{} 
    \item[] Justification: For each experiment we use 3 random seeds and take the average.
    \item[] Guidelines:
    \begin{itemize}
        \item The answer NA means that the paper does not include experiments.
        \item The authors should answer "Yes" if the results are accompanied by error bars, confidence intervals, or statistical significance tests, at least for the experiments that support the main claims of the paper.
        \item The factors of variability that the error bars are capturing should be clearly stated (for example, train/test split, initialization, random drawing of some parameter, or overall run with given experimental conditions).
        \item The method for calculating the error bars should be explained (closed form formula, call to a library function, bootstrap, etc.)
        \item The assumptions made should be given (e.g., Normally distributed errors).
        \item It should be clear whether the error bar is the standard deviation or the standard error of the mean.
        \item It is OK to report 1-sigma error bars, but one should state it. The authors should preferably report a 2-sigma error bar than state that they have a 96\% CI, if the hypothesis of Normality of errors is not verified.
        \item For asymmetric distributions, the authors should be careful not to show in tables or figures symmetric error bars that would yield results that are out of range (e.g. negative error rates).
        \item If error bars are reported in tables or plots, The authors should explain in the text how they were calculated and reference the corresponding figures or tables in the text.
    \end{itemize}

\item {\bf Experiments Compute Resources}
    \item[] Question: For each experiment, does the paper provide sufficient information on the computer resources (type of compute workers, memory, time of execution) needed to reproduce the experiments?
    \item[] Answer: \answerYes{} 
    \item[] Justification: We specify all the computational resources in~\ref{sec:experiment}.
    \item[] Guidelines:
    \begin{itemize}
        \item The answer NA means that the paper does not include experiments.
        \item The paper should indicate the type of compute workers CPU or GPU, internal cluster, or cloud provider, including relevant memory and storage.
        \item The paper should provide the amount of compute required for each of the individual experimental runs as well as estimate the total compute. 
        \item The paper should disclose whether the full research project required more compute than the experiments reported in the paper (e.g., preliminary or failed experiments that didn't make it into the paper). 
    \end{itemize}
    
\item {\bf Code Of Ethics}
    \item[] Question: Does the research conducted in the paper conform, in every respect, with the NeurIPS Code of Ethics \url{https://neurips.cc/public/EthicsGuidelines}?
    \item[] Answer: \answerYes{} 
    \item[] Justification: We conform with the NeurIPS Code of Ethics.
    \item[] Guidelines:
    \begin{itemize}
        \item The answer NA means that the authors have not reviewed the NeurIPS Code of Ethics.
        \item If the authors answer No, they should explain the special circumstances that require a deviation from the Code of Ethics.
        \item The authors should make sure to preserve anonymity (e.g., if there is a special consideration due to laws or regulations in their jurisdiction).
    \end{itemize}

\item {\bf Broader Impacts}
    \item[] Question: Does the paper discuss both potential positive societal impacts and negative societal impacts of the work performed?
    \item[] Answer: \answerNA{} 
    \item[] Justification: Our work is about the theory on online control,
which does not seem to have evident societal impacts.
    \item[] Guidelines:
    \begin{itemize}
        \item The answer NA means that there is no societal impact of the work performed.
        \item If the authors answer NA or No, they should explain why their work has no societal impact or why the paper does not address societal impact.
        \item Examples of negative societal impacts include potential malicious or unintended uses (e.g., disinformation, generating fake profiles, surveillance), fairness considerations (e.g., deployment of technologies that could make decisions that unfairly impact specific groups), privacy considerations, and security considerations.
        \item The conference expects that many papers will be foundational research and not tied to particular applications, let alone deployments. However, if there is a direct path to any negative applications, the authors should point it out. For example, it is legitimate to point out that an improvement in the quality of generative models could be used to generate deepfakes for disinformation. On the other hand, it is not needed to point out that a generic algorithm for optimizing neural networks could enable people to train models that generate Deepfakes faster.
        \item The authors should consider possible harms that could arise when the technology is being used as intended and functioning correctly, harms that could arise when the technology is being used as intended but gives incorrect results, and harms following from (intentional or unintentional) misuse of the technology.
        \item If there are negative societal impacts, the authors could also discuss possible mitigation strategies (e.g., gated release of models, providing defenses in addition to attacks, mechanisms for monitoring misuse, mechanisms to monitor how a system learns from feedback over time, improving the efficiency and accessibility of ML).
    \end{itemize}
    
\item {\bf Safeguards}
    \item[] Question: Does the paper describe safeguards that have been put in place for responsible release of data or models that have a high risk for misuse (e.g., pretrained language models, image generators, or scraped datasets)?
    \item[] Answer: \answerNA{} 
    \item[] Justification: The paper poses no such risks.
    \item[] Guidelines:
    \begin{itemize}
        \item The answer NA means that the paper poses no such risks.
        \item Released models that have a high risk for misuse or dual-use should be released with necessary safeguards to allow for controlled use of the model, for example by requiring that users adhere to usage guidelines or restrictions to access the model or implementing safety filters. 
        \item Datasets that have been scraped from the Internet could pose safety risks. The authors should describe how they avoided releasing unsafe images.
        \item We recognize that providing effective safeguards is challenging, and many papers do not require this, but we encourage authors to take this into account and make a best faith effort.
    \end{itemize}

\item {\bf Licenses for existing assets}
    \item[] Question: Are the creators or original owners of assets (e.g., code, data, models), used in the paper, properly credited and are the license and terms of use explicitly mentioned and properly respected?
    \item[] Answer: \answerYes{} 
    \item[] Justification: We add citations for all datasets we used.
    \item[] Guidelines:
    \begin{itemize}
        \item The answer NA means that the paper does not use existing assets.
        \item The authors should cite the original paper that produced the code package or dataset.
        \item The authors should state which version of the asset is used and, if possible, include a URL.
        \item The name of the license (e.g., CC-BY 4.0) should be included for each asset.
        \item For scraped data from a particular source (e.g., website), the copyright and terms of service of that source should be provided.
        \item If assets are released, the license, copyright information, and terms of use in the package should be provided. For popular datasets, \url{paperswithcode.com/datasets} has curated licenses for some datasets. Their licensing guide can help determine the license of a dataset.
        \item For existing datasets that are re-packaged, both the original license and the license of the derived asset (if it has changed) should be provided.
        \item If this information is not available online, the authors are encouraged to reach out to the asset's creators.
    \end{itemize}

\item {\bf New Assets}
    \item[] Question: Are new assets introduced in the paper well documented and is the documentation provided alongside the assets?
    \item[] Answer: \answerNA{} 
    \item[] Justification: The paper does not release new assets.
    \item[] Guidelines:
    \begin{itemize}
        \item The answer NA means that the paper does not release new assets.
        \item Researchers should communicate the details of the dataset/code/model as part of their submissions via structured templates. This includes details about training, license, limitations, etc. 
        \item The paper should discuss whether and how consent was obtained from people whose asset is used.
        \item At submission time, remember to anonymize your assets (if applicable). You can either create an anonymized URL or include an anonymized zip file.
    \end{itemize}

\item {\bf Crowdsourcing and Research with Human Subjects}
    \item[] Question: For crowdsourcing experiments and research with human subjects, does the paper include the full text of instructions given to participants and screenshots, if applicable, as well as details about compensation (if any)? 
    \item[] Answer: \answerNA{} 
    \item[] Justification: The paper does not involve crowdsourcing nor research with human subjects.
    \item[] Guidelines:
    \begin{itemize}
        \item The answer NA means that the paper does not involve crowdsourcing nor research with human subjects.
        \item Including this information in the supplemental material is fine, but if the main contribution of the paper involves human subjects, then as much detail as possible should be included in the main paper. 
        \item According to the NeurIPS Code of Ethics, workers involved in data collection, curation, or other labor should be paid at least the minimum wage in the country of the data collector. 
    \end{itemize}

\item {\bf Institutional Review Board (IRB) Approvals or Equivalent for Research with Human Subjects}
    \item[] Question: Does the paper describe potential risks incurred by study participants, whether such risks were disclosed to the subjects, and whether Institutional Review Board (IRB) approvals (or an equivalent approval/review based on the requirements of your country or institution) were obtained?
    \item[] Answer: \answerNA{} 
    \item[] Justification: The paper does not involve crowdsourcing nor research with human subjects.
    \item[] Guidelines:
    \begin{itemize}
        \item The answer NA means that the paper does not involve crowdsourcing nor research with human subjects.
        \item Depending on the country in which research is conducted, IRB approval (or equivalent) may be required for any human subjects research. If you obtained IRB approval, you should clearly state this in the paper. 
        \item We recognize that the procedures for this may vary significantly between institutions and locations, and we expect authors to adhere to the NeurIPS Code of Ethics and the guidelines for their institution. 
        \item For initial submissions, do not include any information that would break anonymity (if applicable), such as the institution conducting the review.
    \end{itemize}

\end{enumerate}

\end{document}